\documentclass[11pt,reqno]{amsart}
\usepackage[normalem]{ulem}
\usepackage{cmap}
\usepackage[T1]{fontenc}
\usepackage[utf8]{inputenc}
\usepackage{amsfonts,amssymb,amsthm,amsmath}
\usepackage{url}
\usepackage{enumitem}

\usepackage{secdot}

\usepackage{mathtools}
\mathtoolsset{showonlyrefs}

\usepackage{graphicx}
\usepackage{epstopdf}
\usepackage{a4wide}
\usepackage{xcolor}
\usepackage[colorlinks=true,linktocpage,pdfpagelabels,
bookmarksnumbered,bookmarksopen]{hyperref}
\definecolor{ForestGreen}{rgb}{0.1,0.6,0.05}
\definecolor{EgyptBlue}{rgb}{0.063,0.1,0.6}
\hypersetup{
	colorlinks=true,
	linkcolor=EgyptBlue,         
	citecolor=ForestGreen,
	urlcolor=olive
}

\usepackage[hyperpageref]{backref}

\newtheorem{theorem}{Theorem}
\newtheorem{proposition}[theorem]{Proposition}
\newtheorem{lemma}[theorem]{Lemma}
\newtheorem{corollary}[theorem]{Corollary}

\theoremstyle{definition}

\newtheorem{remark}[theorem]{Remark}

\let\OLDthebibliography\thebibliography
\renewcommand\thebibliography[1]{
	\OLDthebibliography{#1}
	\setlength{\parskip}{1pt}
	\setlength{\itemsep}{1pt plus 0.3ex}
}

\numberwithin{equation}{section}
\numberwithin{theorem}{section}
\numberwithin{equation}{section}
\numberwithin{theorem}{section}

\newcommand{\intO}{\int_\Omega}

\newcommand{\al} {\alpha}
\newcommand{\be} {\beta}
\newcommand{\ga}{\gamma}
\newcommand{\de}{\delta}
\newcommand{\Om}{\Omega}

\title[Reverse Faber-Krahn and Szeg\H{o}-Weinberger type inequalities]{Reverse Faber-Krahn and Szeg\H{o}-Weinberger type inequalities for annular domains under Robin-Neumann boundary conditions}

\author[T.V.~Anoop]{T.V.~Anoop}
\author[V.~Bobkov]{Vladimir Bobkov}
\author[P.~Dr\'abek]{Pavel Dr\'abek}

\address[T.V.~Anoop]{\newline\indent
	Department of Mathematics,
	Indian Institute of Technology Madras, 
 \newline\indent
 Chennai 36, India
}
\email{anoop@iitm.ac.in}

\address[V.~Bobkov]{\newline\indent
	Institute of Mathematics, Ufa Federal Research Centre, RAS,
	\newline\indent 
	Chernyshevsky str. 112, 450008 Ufa, Russia
	\newline\indent 
	Ufa University of Science and Technology,
 \newline\indent
 Zaki Validi str. 32, 450076 Ufa, Russia
}
\email{bobkov@matem.anrb.ru}

\address[P.~Dr\'abek]{\newline\indent
	Department of Mathematics and NTIS, Faculty of Applied Sciences,
	\newline\indent 
	University of West Bohemia, 
 \newline\indent
 Univerzitn\'i 8, 301 00 Plze\v{n}, Czech Republic
}
\email{pdrabek@kma.zcu.cz}

\date{}


\begin{document}
	\begin{abstract}
		Let $\tau_k(\Omega)$ be the $k$-th eigenvalue of the Laplace operator in a bounded domain $\Omega$ of the form $\Omega_{\text{out}} \setminus \overline{B_{\alpha}}$ under the Neumann boundary condition on $\partial \Omega_{\text{out}}$ and the Robin boundary condition with parameter $h \in (-\infty,+\infty]$ on the sphere $\partial B_\alpha$ of radius $\alpha>0$ centered at the origin, the limiting case $h=+\infty$ being understood as the Dirichlet boundary condition on $\partial B_\alpha$.
		In the case $h>0$, it is known that the first eigenvalue $\tau_1(\Omega)$ does not exceed $\tau_1(B_\beta \setminus \overline{B_\alpha})$, where $\beta>0$ is chosen such that $|\Omega| = |B_\beta \setminus \overline{B_\alpha}|$, which can be regarded as a reverse Faber-Krahn type inequality.
		We establish this result for any $h \in (-\infty,+\infty]$.
		Moreover, we provide related estimates for higher eigenvalues under additional geometric assumptions on $\Omega$, which can be seen as Szeg\H{o}-Weinberger type inequalities.
	A few counterexamples to the obtained inequalities for domains violating imposed geometric assumptions are given.
		As auxiliary information, we 
		investigate shapes of eigenfunctions associated with several eigenvalues $\tau_{i}(B_\beta \setminus \overline{B_\alpha})$
		and show that they are nonradial at least for all positive and all sufficiently negative $h$ when $i \in \{2,\ldots,N+2\}$. At the same time, we give numerical evidence that, in the planar case $N=2$, already \textit{second} eigenfunctions can be radial for some $h<0$. 
  The latter fact provides a simple counterexample to the Payne nodal line conjecture in the case of the mixed boundary conditions.
	\end{abstract} 
	\maketitle

	\begin{quote}	
	\setcounter{tocdepth}{1}
	\tableofcontents
	\addtocontents{toc}{\vspace*{0ex}}
	\end{quote}

	\section{Introduction}\label{sec:introduction}

Throughout this article, we denote by $B_r(x_0)$ an open ball of radius $r>0$ centered at $x_0 \in \mathbb{R}^N$. For brevity, we write $B_r$, 
when $x_0$ coincides with the origin of $\mathbb{R}^N$.

Let $\Omega \subset \mathbb{R}^N$, $N \geq 2$, be a set  characterized by the following assumption:
\begin{enumerate}[label={$\mathbf{(A)}$}]
	\item\label{assumption}
	$\Omega$ is a bounded Lipschitz domain of the form
	$\Omega = \Omega_{\text{out}} \setminus \overline{B_\alpha}$, where $\Omega_{\text{out}}$ is an open set, and 
	$B_\alpha$ with some $\alpha>0$ is compactly contained in  $\Omega_{\text{out}}$. 
\end{enumerate}
For the parameter $h \in \mathbb{R} \cup \{+\infty\}$, we consider the mixed Robin-Neumann eigenvalue problem
\begin{equation}\label{eq:D}
	\tag{$\mathcal{EP}$}
	\left\{
	\begin{aligned}
		-\Delta u &= \tau u &&\text{in}~\Omega,\\
		\frac{\partial u}{\partial \eta} &+ h u = 0 &&\text{on}~\partial B_\alpha,\\
		\frac{\partial u}{\partial \eta} &= 0 &&\text{on}~\partial \Omega_{\text{out}},
	\end{aligned}
	\right.
\end{equation}
where $\eta$ stands for the outward unit normal vector to $\partial \Omega$.
In the formal limiting case $h = +\infty$, we assume the zero Dirichlet boundary condition $u=0$ on $\partial B_\alpha$.

It follows from the general theory of compact self-adjoint operators (see, e.g., a condensed exposition in \cite[Section~4.2]{BFK} on a closely related problem) that the spectrum of  \eqref{eq:D} consists of a discrete sequence of eigenvalues $\{\tau_k(\Omega)\}$ accumulating at infinity:
$$
\tau_1(\Omega) < \tau_2(\Omega) \leq \tau_3(\Omega) \leq \ldots \leq \tau_k(\Omega) \leq \ldots, 
\quad 
\tau_k(\Omega) \to +\infty ~~\text{as}~~ k \to +\infty,
$$
each of which can be characterized via the classical minimax principles, e.g.,
\begin{equation}\label{eq:tauk}
	\tau_{k}(\Omega) 
	= 
	\min_{X\in\mathcal{X}_k} 
	\max_{u \in X\setminus \{0\}} \frac{\int_{\Omega} |\nabla u|^2 \, dx + h \int_{\partial B_\alpha} u^2 \,dS}{\int_{\Omega} u^2 \, dx},
	\quad
	k \in \mathbb{N},
\end{equation} 
where $\mathcal{X}_k$ is the collection of all $k$-dimensional subspaces of the Sobolev space
\begin{equation}\label{eq:sobolev}
\tilde{H}^1(\Omega)
=
\left\{
\begin{array}{ll}
	H^1(\Omega) &\text{ if } h<+\infty, \\  
	\{v\in H^1(\Omega):~ v=0 ~\text{on}~ \partial B_\alpha \} & \text{ if } h=+\infty,
\end{array}
\right.
\end{equation}
the equality $v=0$ being classically understood in the sense of traces, and we set 
$h \int_{\partial B_\alpha} u^2 \,dS = 0$ in the Dirichlet case $h=+\infty$.
In particular, the first eigenvalue is simple and can be defined as
\begin{equation}\label{eq:tau1}
	\tau_1(\Omega)
	=
	\min\left\{
	\frac{\int_{\Omega} |\nabla u|^2 \,dx + h \int_{\partial B_\alpha} u^2 \,dS}{\int_{\Omega} u^2 \,dx}:~
	u \in \tilde{H}^1(\Omega) \setminus\{0\}
	\right\}.
\end{equation}
By taking $u=\text{const}$ as a trial function in \eqref{eq:tau1} when $h<+\infty$, we see that $\tau_1(\Omega)<0$ for $h<0$, and $\tau_1(\Omega)=0$ for $h=0$ (i.e., in the case of the purely Neumann boundary conditions). Moreover, it is not hard to deduce from \eqref{eq:tau1} that $\tau_1(\Omega)>0$ for $h \in (0,+\infty]$.

Assume, for a moment, that the Robin boundary condition $\frac{\partial u}{\partial \eta} + h u = 0$ is imposed on the \textit{whole} boundary of a general, bounded, sufficiently smooth domain $\Omega$. 
The investigation of properties of the corresponding eigenvalues $\{\tau_k(\Omega)\}$ has a rich history with various fruitful results, many of which were obtained only recently.
Let us briefly comment on several results according to the particular choice of the boundary conditions, and we refer the reader to the comprehensive overviews in \cite{BFK,FL1,henrot} for further details of both mathematical and historical natures.
\begin{enumerate}[wide=12pt,itemsep=5pt]
	\item In the Dirichlet case $h=+\infty$, the fundamental estimate for $\tau_1(\Omega)$ is given by the Faber-Krahn inequality $\tau_1(\Omega) \geq \tau_1(B)$, where $B$ is an open ball of the same measure as $\Omega$.
	Its extension to $\tau_2(\Omega)$, which is sometimes referred to as the Hong-Krahn-Szeg\H{o} inequality (see, e.g., \cite{BF}), reads as $\tau_2(\Omega) \geq \tau_2(\tilde{B})$, where $\tilde{B}$ is a disjoint union of two open balls, each of radius $|\Omega|/2$.
	Notice that $\tau_2(\tilde{B}) = 2^{2/N}\tau_1(B)$.
	\item In the Neumann case $h=0$, the first eigenvalue $\tau_1(\Omega)$ is zero. Still, a counterpart of the Faber-Krahn inequality, known as the Szeg\H{o}-Weinberger inequality, holds for the first \textit{nonzero} eigenvalue and has the form $\tau_2(\Omega) \leq \tau_2(B)$.
	Moreover, \textsc{Girouard,  Nadirashvili, \&  Polterovich} \cite{GNP} (the case $N=2$) and \textsc{Bucur \& Henrot} \cite{BH} (the case $N\geq 2$) generalized this results to the second nonzero eigenvalue by proving that $\tau_3(\Omega) \leq \tau_3(\tilde{B})$.
	Under additional symmetry assumptions on $\Omega$, such as the symmetry of order $q$, inequalities of similar type were obtained for higher eigenvalues by \textsc{Hersch} \cite{hersh2} and \textsc{Ashbaugh \& Benguria} \cite{AB}, and we also refer to \textsc{Enache \& Philippin} \cite{EF1, EF2} for related results.
	\item In the Robin case $h \in (0,+\infty)$, 
	\textsc{Bossel} \cite{bossel} (the case $N=2$) and \textsc{Daners} \cite{daners} (the case $N\geq 2$) proved the validity of the original Faber-Krahn inequality $\tau_1(\Omega) \geq \tau_1(B)$. 
	As for the second eigenvalue, \textsc{Kennedy} shown in \cite{kennedy1} that the Hong-Krahn-Szeg\H{o} inequality $\tau_2(\Omega) \geq \tau_2(\tilde{B})$ is also satisfied.
	\item In the Robin case $h \in (-\infty,0)$, the situation is less clear.
	The existence of $h^*=h^*(|\Omega|) < 0$ such that $\tau_1(\Omega) \leq \tau_1(B)$ for any $h \in [h^*,0)$ was proved by \textsc{Freitas \& Krej\v{c}i\v{r}\'ik} \cite{FK} in the case $N=2$. 
 Moreover, it was shown that this inequality is not generally true for any $N \geq 2$, since it is reversed when $\Omega$ is a spherical shell and $|h|$ is sufficiently large.
	More recently, it was proved by \textsc{Freitas \& Laugesen} \cite{FL1} that $\tau_2(\Omega) \leq \tau_2(B)$ for any $h \in [-(N+1)/(RN), 0)$, where $R$ is the radius of the ball $B$.
The maximization of the third eigenvalue $\tau_3(\Omega)$ under an additional normalization assumption on $h$ was considered by \textsc{Girouard \& Laugesen} \cite{GL} in the case $N=2$.
\end{enumerate}

We also refer to \cite{AFKen,AFK,laug1} for some numeric and analytic results for higher eigenvalues of the Robin problem, which lead to several important conjectures.

The fundamental inequalities mentioned above are universal with respect to the domain and they are
not sensitive to a possible presence of ``holes'' in $\Omega$ as in the assumption \ref{assumption}.
However, for certain classes of domains, refined estimates for lower frequencies $\tau_k(\Omega)$ are known if the topology of $\Omega$ is taken into account. 
In this respect, we refer to \textsc{Hersch} \cite[Section~3]{hersh} and references therein for results in the purely Dirichlet case and to \textsc{Exner \& Lotoreichik} \cite{EL} 
for the purely Robin case.
In the purely Neumann case, it was proved by the present authors in \cite{ABD1} that, under additional symmetry assumptions on $\Omega$, which will be discussed below (similar to that used in \cite{AB,hersh2}), the  Szeg\H{o}-Weinberger type inequality
\begin{equation}\label{SW1high}
	\tau_i(\Omega_{\text{out}} \setminus \overline{B_\alpha}) \leq \tau_i(B_\beta \setminus \overline{B_\alpha})
\end{equation}
is valid for some higher indices $i \geq 2$, 
where the ball $B_\beta$ is centered at the origin and chosen in such a way that $|\Omega_{\text{out}} \setminus \overline{B_\alpha}| = |B_\beta \setminus \overline{B_\alpha}|$.
The inequality \eqref{SW1high} gives a better bound
than the classical Szeg\H{o}-Weinberger inequality $\tau_2(\Omega) \leq \tau_2(B)$, as it follows from \cite[Corollary~1.7]{ABD1}.

Returning back to the original problem \eqref{eq:D} with the mixed boundary conditions, several results in this direction are also known.
\textsc{Della Pietra \& Piscitelli} proved in \cite{DPP}, as a particular case of a more general result for the $p$-Laplacian and a convex inner ``hole'', that the reverse Faber-Krahn (or, equivalently, the reverse Bossel-Daners) type inequality
\begin{equation}\label{SW1}
	\tau_1(\Omega_{\text{out}} \setminus \overline{B_\alpha}) \leq \tau_1(B_\beta \setminus \overline{B_\alpha})
\end{equation}
is satisfied provided $h \in (0,+\infty)$.
In the Dirichlet-Neumann case $h = +\infty$, the same inequality \eqref{SW1} follows from more general results established by \textsc{Hersch} \cite{hersh0} (see also \cite{hersh}) for $N=2$ and by \textsc{Anoop \& Kumar} \cite{AK} for $N \geq 2$.

In our first result, we show that the reverse Faber-Krahn (Bossel-Daners) inequality \eqref{SW1} is satisfied also for negative values of the Robin parameter $h$, and we are able to characterize the equality case.
We notice that, unlike \cite{AK,DPP,hersh0,hersh}, we deal only with ``holes'' of spherical shape, and, unlike \cite{AK,DPP}, only with the Laplace operator. These restrictions are due to our proofs based on the approach of \textsc{Weinberger} \cite{weinberger}. 
We discuss it at the end of this section, and here we  mention that this method differs from those used in \cite{AK,DPP,hersh0}.
\begin{theorem}[Reverse Faber-Krahn (Bossel-Daners) inequality]\label{thm0}
	Let $h \in \mathbb{R} \cup \{+\infty\}$ and $h \neq 0$.
	Let $\Omega = \Omega_{\textnormal{out}} \setminus \overline{B_\alpha}$ satisfy the assumption \ref{assumption}, and let $\beta>\alpha$ be such that $|B_\beta \setminus \overline{B_\alpha}|=|\Omega_{\textnormal{out}} \setminus \overline{B_\alpha}|$.
	Then
	\begin{equation}\label{eq:RFK}
		\tau_{1}(\Omega_{\textnormal{out}} \setminus \overline{B_\alpha}) \leq \tau_1(B_\beta \setminus \overline{B_\alpha}).
	\end{equation}
	Moreover, equality holds in \eqref{eq:RFK} if and only if  $\Omega = B_\beta \setminus \overline{B_\alpha}$.
\end{theorem}

We have the following simple corollary of Theorem~\ref{thm0} saying that concentric spherical shells have the largest first eigenvalue among generally eccentric spherical shells of the same measure.
As in the case of Theorem~\ref{thm0}, this result is known for $h=+\infty$, see \cite[Theorem~1.2]{AK}, and we refer to \cite{AK} for a historical overview of related results in this direction.
\begin{corollary}
	Let $h \in \mathbb{R} \cup \{+\infty\}$, $h \neq 0$, and $0<\alpha<\beta$.
	Let $x_0 \in \mathbb{R}^N$ be such that $B_\al(x_0)$ is compactly contained in $B_\beta$.
	Then
	\begin{equation}\label{eq:RFKrad}
		\tau_{1}(B_\beta \setminus \overline{B_\alpha(x_0)}) \leq \tau_1(B_\beta \setminus \overline{B_\alpha}).
	\end{equation}
	Moreover, equality holds in \eqref{eq:RFKrad} if and only if $x_0$ is the origin of $\mathbb{R}^N$.
\end{corollary}

Our second aim is to develop the results obtained in \cite{ABD1} on the Szeg\H{o}-Weinberger type inequalities \eqref{SW1high} by covering the case of nonzero Robin parameter $h$ and also dealing with eigenvalues of even higher indices than in \cite{ABD1}
assuming either sufficient negativity of $h$ or sufficient ``thinness'' of $\Omega$.
First, let us define two simple symmetry classes of domains used in \cite{ABD1,AB,hersh2}.
\begin{enumerate}[wide=12pt,itemsep=5pt]
	\item A domain $\Omega$ is called \textit{symmetric of order $q \in \mathbb{N}$}
	if $R^{2\pi/q}_{i,j} \Omega = \Omega$ for any $1 \leq i < j \leq N$, where $R^{2\pi/q}_{i,j}$ denotes the rotation (in the anticlockwise direction with respect to the origin) by angle $2\pi/q$ in the coordinate plane $(x_i,x_j)$. 
	\item A domain $\Omega$ is called \textit{centrally symmetric} whenever $x \in \Omega$ if and only if $-x \in \Omega$.
\end{enumerate}

Since the spectrum of \eqref{eq:D} does not depend on isometries and translations of $\Omega$, these symmetry classes are defined up to corresponding transformations of $\Omega$. 
The inclusions between these classes and some of their properties are discussed in detail in \cite[Section~5]{ABD1}, see also \cite[Remark~1.4]{ABD1}.
In particular, in the planar case $N=2$, the symmetry of order $2$ is equivalent to the central symmetry; 
in the case of even dimensions $N = 4, 6, \ldots$, the symmetry of order
$2$ always implies the central symmetry, but not vice versa; in the case of odd dimensions $N=3,5,\ldots$, these two symmetry classes are independent.
Moreover, in the case $N=2$, nonradial domains with the symmetry of order $q$ exist for any $q \in \mathbb{N}$. In contrast, in the higher-dimensional case $N \geq 3$, such domains exist only for $q=1,2,4$.

Next, we formulate our main result in the planar case $N=2$, see Figure~\ref{fig:thm1} for a schematic graph of assumptions on $\alpha$ and $h$, where $\beta$ is assumed to be fixed.
\begin{theorem}\label{thm1-ext}
    Let $N=2$, $\beta>0$, and $\kappa \geq 1$ be an integer.
    Then for any $\alpha \in (0,\beta)$ there exists $\bar{h}_\kappa \in \mathbb{R} \cup \{+\infty\}$ such that for any $h \in (-\infty,\bar{h}_\kappa]$ and any $\Omega = \Omega_{\textnormal{out}} \setminus \overline{B_\alpha}$ satisfying the assumption \ref{assumption}, symmetric of order $2^{\kappa}$, and such that $|B_\beta \setminus \overline{B_\alpha}|=|\Omega_{\textnormal{out}} \setminus \overline{B_\alpha}|$,  
    we have
    \begin{equation}\label{eq:SWhigh}
		\tau_{i}(\Omega_{\textnormal{out}} \setminus \overline{B_\alpha}) \leq \tau_{i}(B_\beta \setminus \overline{B_\alpha})
		\quad \text{for} \quad
	i = 2,3,\dots,2^\kappa.
  \end{equation}
 Moreover, there exists $\alpha_\kappa \in (0,\beta)$, which depends only on $\beta$ and $\kappa$, such that $\bar{h}_\kappa = +\infty$ for any $\alpha \in [\alpha_\kappa,\beta)$.
 In addition, 
 for any $\alpha \in (0,\alpha_2)$ there exists $h_2 < 0$ such that the following assertions hold:
 \begin{enumerate}[label={\rm(\roman*)}]
     \item If $\kappa=1,2$, then \eqref{eq:SWhigh} holds for any $h \in [h_2,+\infty]$.
     \item If $\kappa=3$, then \eqref{eq:SWhigh} with $i \leq 5$ holds for any $h \in [h_2,+\infty]$.
 \end{enumerate}

\noindent    Furthermore, equality holds in \eqref{eq:SWhigh} if and only if $\Omega = B_\beta \setminus \overline{B_\alpha}$.
\end{theorem}

Now we provide a result for higher dimensions, i.e., $N \geq 3$.

\begin{theorem}\label{thm1}
    Let $N \geq 3$ and $\beta>0$. 
    Then for any $\alpha \in (0,\beta)$ 
    there exist 
    $h_2 < 0$ and $\bar{h}_2 \in \mathbb{R} \cup \{+\infty\}$ such that  
    for any $h \in (-\infty,\bar{h}_2] \cup [h_2,+\infty]$
    and
    any $\Omega = \Omega_{\textnormal{out}} \setminus \overline{B_\alpha}$ satisfying the assumption \ref{assumption} and such that $|B_\beta \setminus \overline{B_\alpha}|=|\Omega_{\textnormal{out}} \setminus \overline{B_\alpha}|$, 
    the following assertions hold:
	\begin{enumerate}[label={\rm(\arabic*)}]
		\item\label{thm1:1} 
		If $\Omega$ is symmetric of order $2$ or centrally symmetric, then 
		\begin{equation}\label{eq:SWND}
			\tau_{2}(\Omega_{\textnormal{out}} \setminus \overline{B_\alpha}) \leq \tau_2(B_\beta \setminus \overline{B_\alpha}).
		\end{equation}
		\item\label{thm1:2} 
		If $\Omega$ is symmetric of order $4$, then
		\begin{equation}\label{eq:SWNDx}
			\tau_{i}(\Omega_{\textnormal{out}} \setminus \overline{B_\al}) \leq \tau_i(B_\beta \setminus \overline{B_\alpha})=\tau_2(B_\beta \setminus \overline{B_\alpha}) \quad \text{for} \quad  i=2,3,\ldots, N+1,
		\end{equation}
and	\begin{equation}\label{eq:SWNDy}
			\tau_{N+2}(\Omega_{\textnormal{out}} \setminus \overline{B_\al}) \leq \tau_{N+2}(B_\beta \setminus \overline{B_\alpha}).
		\end{equation}
	\end{enumerate}
Moreover, there exists $\alpha_2 \in (0,\beta)$, which depends only on $N$ and $\beta$, such that  $\bar{h}_2 = +\infty$ for any $\alpha \in [\alpha_2,\beta)$. 

Furthermore, equality holds in \eqref{eq:SWND}, \eqref{eq:SWNDx}, \eqref{eq:SWNDy} 
 %\eqref{eq:SWNDz} 
	if and only if $\Omega = B_\beta \setminus \overline{B_\alpha}$.
\end{theorem}

\begin{figure}[ht!]
	\centering
	\includegraphics[width=0.6\linewidth]{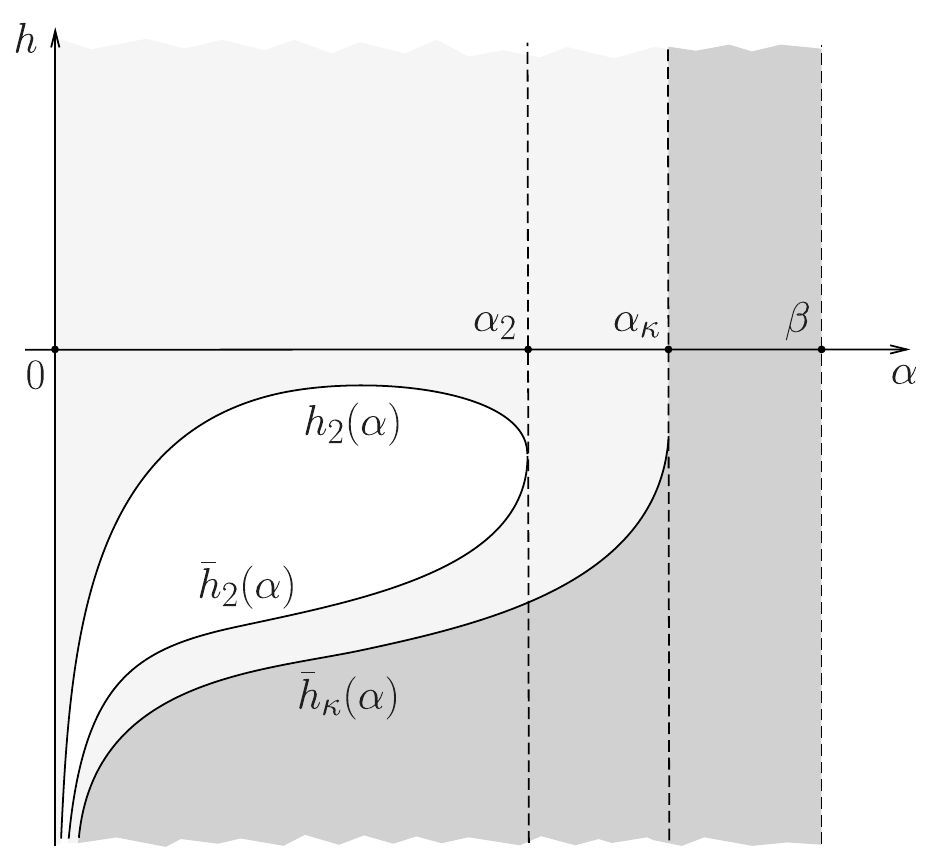}\\
	\caption{A schematic graph of the assumptions of Theorem~\ref{thm1-ext} (dark gray, $\kappa \geq 3$) and Theorems~\ref{thm1-ext},~\ref{thm1} (the union of light and dark grays, $\kappa=2$)  on the $(\alpha,h)$-plane.}
	\label{fig:thm1}
\end{figure}

We do not know whether the inequalities 
\eqref{eq:SWhigh}, \eqref{eq:SWND}, \eqref{eq:SWNDx}, \eqref{eq:SWNDy} remain valid for all $h \in \mathbb{R} \cup \{+\infty\}$ and $\alpha \in (0,\beta)$. 
The assumptions on $h$ and $\alpha$ in Theorems~\ref{thm1-ext},~\ref{thm1} (see Figure~\ref{fig:thm1}) occur from the method of proof requiring eigenfunctions associated with 
$\tau_i(B_\beta \setminus \overline{B_\alpha})$ to be nonradial.
In Remark~\ref{rem:radial} below, we discuss that, in general, already \textit{second} eigenfunctions in the planar ring
$B_\beta \setminus \overline{B_\alpha}$ can be radial for certain $h<0$ and $0<\alpha<\beta$ (which results in the restrictions on $h$ in Theorems~\ref{thm1-ext},~\ref{thm1}).
This curious fact indicates a geometrically simple counterexample to the well-known Payne conjecture asserting that nodal lines of second eigenfunctions have to intersect the boundary of $\Omega$. 
We refer to \cite{ken1} for a related discussion.

Let us note that several works mentioned above (see \cite{AK,FK,hersh0,hersh} and also \cite{PPT,PW}) further provide estimates for $\tau_1(\Omega)$ in the case of reversed mixed boundary conditions, such as the Neumann conditions on the inner boundary and the Robin or Dirichlet conditions on the outer boundary, and under more general assumptions on the geometry of the domain.
Our choice of the boundary conditions in \eqref{eq:D} and the spherical shape of the ``hole'' 
is dictated by the method of the proof, which is common for all our results -- Theorems~\ref{thm0}, \ref{thm1-ext}, and \ref{thm1}. 
The method
is built upon and develops the original arguments of \textsc{Weinberger}  \cite{weinberger}.
More precisely, we use an orthogonal basis of eigenfunctions of \eqref{eq:D} in the spherical shell $B_\be\setminus \overline{B_\al}$ to construct appropriate finite-dimensional subspaces of $H^1(\Omega)$ for a minimax characterization of $\tau_i(\Om)$. 
This is done by extending eigenfunctions from $B_\be\setminus \overline{B_\al}$ to $\Omega = \Omega_{\text{out}} \setminus \overline{B_\al}$. 
In the case of higher eigenvalues, symmetry assumptions on $\Omega$ help to guarantee certain orthogonality of basis elements of these trial subspaces. 
In \cite{weinberger} and \cite{BH}, dealing with the Neumann eigenvalues $\tau_2(\Omega)$ and $\tau_3(\Omega)$, respectively, this  orthogonality is justified for any domain regardless of its symmetry, but the obtained upper bounds are not sensitive to the inclusion of ``holes''.
When dealing with ``holes'', the required orthogonality is not generally true for arbitrary domains even in the case of $\tau_2(\Omega)$, as can be seen from counterexamples in \cite[Section~4]{ABD1} or Section~\ref{sec:counterexample} below.
(Note that the orthogonality can be \textit{assumed}, 
which leads to more general but less constructive assumptions on $\Omega$, see a discussion in Remark~\ref{rem:symmetry} below.)
A key inequality for the proofs of Theorems~\ref{thm1-ext} and~\ref{thm1} is given by Proposition~\ref{prop:bound} and it is used to show that the constructed trial subspaces deliver required upper bounds on $\tau_i(\Omega)$. 
We were not able to extend this inequality for ``holes'' of a more general shape, but the possibility of such  extension does not seem hopeless to us.
Finally, let us mention that, to the best of our knowledge, applications of \textsc{Weinberger}'s approach to eigenvalue problems with non-Neumann boundary conditions are rare in the literature, and we can only refer to \cite{FL1} for an estimate of $\tau_2(\Omega)$ with negative $h$.

%\smallskip
The remaining part of this work has the following structure.
Section~\ref{sec:spectrum} contains some auxiliary results on the structure and properties of the spectrum  $\{\tau_k(B_\beta \setminus \overline{B_\alpha})\}$, most of the proofs being placed in Appendix~\ref{sec:appendix}.
Section~\ref{sec:proof2} is devoted to the proofs of our main results -- Theorems~\ref{thm0}, \ref{thm1-ext},  \ref{thm1}.
In Section~\ref{sec:counterexample}, we discuss the violation of the obtained inequalities for domains that do not satisfy the required symmetry assumptions.
Section~\ref{sec:final} contains some concluding remarks.
Finally, in Appendix~\ref{sec:appendix2}, we provide a characterization of $\tau_{k}(\Omega)$ which is slightly more convenient for the proofs of the main results than \eqref{eq:tauk}.

\section{Spectrum of spherical shells}\label{sec:spectrum}

In this section, we collect several facts on the structure of eigenvalues and eigenfunctions of the problem \eqref{eq:D} in the spherical shell $B_\beta \setminus \overline{B_\alpha} \subset \mathbb{R}^N$ (with $0<\alpha<\beta$ and $N \geq 2$) needed for our purposes. 
Thanks to the possibility of separation of variables, preliminary results of this section are classical, and we include them for the consistency of exposition.
Some related considerations can be found in, e.g.,
\cite{ABD1,FK,PPT}.
%\cite[Section~2]{ABD1} in the Dirichlet-Neumann case ($h=+\infty)$, in \cite[Sections~3, 4]{FK}, \cite{PPT} in the Robin-Robin and Neumann-Robin cases.

Hereinafter, we will use the notation $\mathbb{N}_0 = \{0,1,2,\dots\}$, whereas $\mathbb{N} = \{1,2,\dots\}$.
Separating the variables, 
one can search for the basis of eigenfunctions of \eqref{eq:D} in $B_\beta \setminus \overline{B_\alpha}$ in the form
\begin{equation}\label{eq:sepvar}
\phi(x) = v(|x|) \, \xi \left(\frac{x}{|x|}\right), ~x\ne 0.
\end{equation}
Here, $\xi$ is an eigenfunction of the Laplace--Beltrami operator $\Delta_{S^{N-1}}$ on the unit sphere $S^{N-1}$ corresponding to the eigenvalue $-l(l+N-2)$. It is well-known that the multiplicity of this eigenvalue is 
\begin{equation}\label{eq:Lambdal}
\Lambda_l
=
\binom{l+N-1}{N-1} - \binom{l+N-3}{N-1},
\end{equation}
see, e.g., \cite[Sections~22.3, 22.4]{shubin}.
In particular, 
if $l=0$, then $\Lambda_0=1$, $\xi$ is a nonzero constant function, and hence the eigenfunction $\phi$ is radial, see \eqref{eq:sepvar}.
Moreover, if $N =2$, then $\Lambda_l=2$ for any $l \in \mathbb{N}$, and if $N \geq 3$, then $\Lambda_1=N$.

The function $v$ in \eqref{eq:sepvar} is an eigenfunction of the Sturm-Liouville eigenvalue problem (\textit{SL problem}, for short) for the equation
\begin{equation}\label{eq:ode}
-v'' - \frac{N-1}{r} v' + \frac{l(l+N-2)}{r^2} v = \tau v, 
\quad r \in (\alpha,\beta),
\end{equation}
with the mixed Robin-Neumann boundary conditions
\begin{equation}\label{eq:bcM}
-v'(\alpha) + hv(\alpha)=0 \quad \text{and} \quad v'(\beta)=0.
\end{equation}
In the formal limiting case $h=+\infty$, we assume that the Dirichlet boundary condition $v(\alpha)=0$ is imposed.

By the standard Sturm--Liouville theory (see, e.g., \cite[Section~13]{weidmann}), 
for any fixed $l \in \mathbb{N}_0$ the spectrum of the problem \eqref{eq:ode}-\eqref{eq:bcM} consists of a sequence of eigenvalues $\{\tau_{l,j}\}_{j \in \mathbb{N}}$ with the properties
\begin{equation}\label{eq:row}
\tau_{l,1} 
< 
\tau_{l,2} 
< 
\ldots < \tau_{l,j} < \ldots, 
\quad 
\tau_{l,j} \to +\infty ~\text{as}~ j \to +\infty,
\end{equation}
any $\tau_{l,j}$ is simple, and the associated eigenfunction vanishes exactly $j-1$ times in $(\al,\be)$. Let $H^1((\al,\be);r^{N-1})$ be the weighted Sobolev space with the weight $r^{N-1}.$ 
We define, for brevity, the space
$$
\tilde{H}^1((\al,\be);r^{N-1})=\left\{\begin{array}{ll}
   H^1((\al,\be);r^{N-1}) &\text{ if } h<+\infty, \\  
   \{v\in H^1((\al,\be);r^{N-1}):~  v(\al)=0\} & \text{ if } h=+\infty. 
\end{array}\right.
$$
By rewriting the equation \eqref{eq:ode} as
\begin{equation}\label{eq:ode3}
-(r^{N-1}v')' + l(l+N-2) r^{N-3} v = \tau r^{N-1} v, 
\quad r \in (\alpha,\beta),
\end{equation}
we see that every eigenvalue of the SL problem \eqref{eq:ode}-\eqref{eq:bcM} is a critical value of the Rayleigh quotient
\begin{equation}\label{eq:RQ}
R_l(v)=\frac{\int_\al^\be \left[(v'(r))^2+\frac{l(l+N-2)}{r^2}v^2(r)\right]r^{N-1}\,dr + h \alpha^{N-1} v^2(\alpha)}{\int_\al^\be v^2(r)r^{N-1}\,dr}
\end{equation}
over $\tilde{H}^1((\al,\be);r^{N-1}) \setminus\{0\}$, where we set $h \alpha^{N-1} v^2(\alpha) = 0$ 
in the case $h=+\infty$. 
Consequently, each eigenvalue $\tau_{l,j}$ can be characterized by, e.g., the Courant--Fischer minimax formula as
\begin{equation}\label{eq:taudef}
\tau_{l,j}
=
\min_{Z\in \mathcal{Z}_j}\max_{u\in Z\setminus \{0\}}R_l(u),
\end{equation}
where $\mathcal{Z}_j$ is the collection of all $j$-dimensional subspaces of $\tilde{H}^1((\al,\be);r^{N-1}).$ 
In particular, 
\begin{equation}\label{eq:taudef-first}
	\tau_{l,1}
	=
	\min_{u\in \tilde{H}^1((\al,\be);r^{N-1})\setminus \{0\}}R_l(u).
\end{equation}
Noting that $R_l(v)$ is (strictly) increasing with respect to $l$ for any nonzero $v$, it can be proved that for each fixed $j\in \mathbb{N}$,
\begin{equation}\label{eqn:col}
\tau_{0,j}
< 
\tau_{1,j}
< 
\ldots < \tau_{l,j} < \ldots, 
\quad 
\tau_{l,j} \to +\infty ~~\text{as}~~ l \to +\infty.
\end{equation}
Let us also mention that any Robin-Neumann eigenvalue $\tau_{l,j}$ converges to the Dirichlet-Neumann  eigenvalue with the same indices as $h \to +\infty$, see \cite[Theorem~(a):(i)]{EMZ} and Lemma~\ref{lem:tau-derivative} below.

The following lemma, which we prove in Appendix~\ref{sec:appendix}, describes the behavior of the first eigenvalues and eigenfunctions of the SL problem \eqref{eq:ode}-\eqref{eq:bcM}, and it
is needed for the proof of Proposition  \ref{prop:bound} below.
In \cite[Lemma~2.7]{ABD1}, a related result for the Neumann case $h=0$ can be found.
%make sure that l=0 is ok!!!!!
\begin{lemma}\label{lem:v}
	Let $h \in \mathbb{R} \cup \{+\infty\}$, $0<\alpha<\beta$, 
	and $l \in \mathbb{N}_0$.
	Let $v$ be a positive  eigenfunction corresponding to the eigenvalue $\tau_{l,1}$ of the SL problem \eqref{eq:ode}-\eqref{eq:bcM}.
	Then
	\begin{equation}\label{eq:long}
		\left(\frac{l(l+N-2)}{r^2}
		-
		\tau_{l,1}\right)v^2(r)
		\geq
		\left(\frac{l(l+N-2)}{\beta^2}
		-
		\tau_{l,1}\right)v^2(\beta),
		\quad r \in (\alpha,\beta),
	\end{equation}
	and
	the following assertions are satisfied:
	\begin{enumerate}[label={\rm(\arabic*)}]
		\item\label{eq:long:1} Let $l=0$.
		\begin{enumerate}[label={\rm(\roman*)}]
			\item\label{eq:long:1:i} If $h > 0$, then $v'>0$ in $(\alpha,\beta)$ and $\tau_{0,1}>0$.
			\item\label{eq:long:1:ii} If $h = 0$, then $v'=0$ in $(\alpha,\beta)$ and $\tau_{0,1}=0$.
			\item\label{eq:long:1:iii} If $h < 0$, then $v'<0$ in $(\alpha,\beta)$ and $\tau_{0,1}<0$.
		\end{enumerate}
		\item\label{eq:long:2} Let $l \geq 1$. 
		\begin{enumerate}[label={\rm(\roman*)}]
  \item\label{eq:long:2:iii}
			There exists $h_1<0$ such that $\tau_{l,1}>0$ if $h > h_1$, $\tau_{l,1}=0$ if $h=h_1$, and $\tau_{l,1}<0$ if $h < h_1$.
		Moreover, $\tau_{l,1}$ continuously increases with respect to $h$.
			\item\label{eq:long:2:i} If $h \geq 0$, then $v'> 0$ in $ (\alpha,\beta)$.
			\item\label{eq:long:2:ii} 
			There exists $h_0<0$ with the following properties: 
			\begin{enumerate}[label={\rm(\alph*)}]
			\item\label{eq:long:2:ii:a} If $h \in (h_0,0)$, then there exists $\gamma \in (\alpha,\beta)$ such that
			$v' < 0$ in $(\alpha,\gamma)$, $v'(\gamma)=0$, and $v' > 0$ in $(\gamma,\beta)$.
			\item\label{eq:long:2:ii:b} If $h \leq h_0$, then $v'<0$ in $(\alpha,\beta)$.
			\end{enumerate}
		\end{enumerate}
	\end{enumerate}
\end{lemma}

Let us now provide a result that plays a key role in the proof of Theorems~\ref{thm1-ext} and~\ref{thm1}.
This result is a suitable (but not direct) generalization of \textsc{Weinberger}'s arguments \cite[Eqs.~(2.11)-(2.17)]{weinberger}, since they are not generally applicable to non-Neumann problems and domains with ``holes'' because of the behavior of the corresponding eigenfunctions.
Its proof is reminiscent of that of \cite[Proposition~2.8]{ABD1} which covers only the Neumann case $h=0$ but deals with ``holes'' of a more general shape.

%make sure that l=0 is ok!!!
\begin{proposition}\label{prop:bound}
		Let $\Omega$ satisfy the assumption \ref{assumption} and let $\beta>\alpha$ be such that
		$|\Omega| = |B_\beta \setminus B_\alpha|$.
		Let $h \in \mathbb{R} \cup \{+\infty\}$, $l \in \mathbb{N}_0$, and let $v$ be a positive  eigenfunction corresponding to the eigenvalue $\tau_{l,1}$ of the SL problem \eqref{eq:ode}-\eqref{eq:bcM}. 
		Define
		\begin{equation}\label{eq:G}
        G(r)
        =
        \left\{
        \begin{aligned}
      &v(r) 	&&\text{for}~ r \in [\alpha,\beta), \\
      &v(\beta) &&\text{for}~ r \geq \beta.
        \end{aligned}
        \right.
        \end{equation}
		Then
	\begin{equation}\label{eq:bound1}
	 \frac{\int_{\Omega}
	\left(
	(G'(|x|))^2 + \frac{l(l+N-2)}{|x|^2}G^2(|x|)
	\right)
	dx + h \int_{\partial B_\al}  G^2(|x|) \, dS}
{\int_{\Omega} G^2(|x|) \, dx}
\leq \tau_{l,1}.
	\end{equation} 
Moreover, equality holds in \eqref{eq:bound1} if and only if either $\Omega = B_\beta \setminus \overline{B_\alpha}$ or $h,l=0$.
\end{proposition}
\begin{proof}
For brevity, we introduce the following notation: 
\begin{align*}
\label{eq:notationH}
H(r)
&=
(G'(r))^2+\frac{l(l+N-2)}{r^2}G^2(r),
\quad r>0,\\
\widetilde{G} 
&= 
h \int_{\partial B_\al}  G^2(|x|) \, dS
\equiv
h \alpha^{N-1} |\partial B_1| \, G^2(\alpha).
\end{align*}
With this notation, the desired inequality \eqref{eq:bound1} is equivalent to
\begin{equation}\label{eq:mainineq1}
\int_{\Omega}H(|x|) \,dx+ \widetilde{G}
\leq \tau_{l,1} \int_{\Omega} G^2(|x|) \, dx.
\end{equation}

Let us represent $\Omega$ as a disjoint union of its intersections with $B_\beta$ and the complement $B_\beta^c$:
\begin{equation}\label{eq:Omdec}
\Om
= 
[\Om \cap B_\beta] \cup [\Om \cap B_\beta^c].
\end{equation}
In the same way, recalling that $B_\al$ is the ''hole'' in $\Omega$, we get
\begin{equation}\label{eq:Shelldec}
 B_\beta \setminus \overline{B_\alpha} = 
 [\Om \cap (B_\beta \setminus \overline{B_\alpha})]
 \cup 
 [\Om^c \cap (B_\beta \setminus \overline{B_\alpha})]
 =
 [\Om \cap B_\beta]
 \cup 
 [\Om^c \cap (B_\beta \setminus \overline{B_\alpha})].
\end{equation}
Since $|\Om|=| B_\beta \setminus \overline{B_\alpha}|$ by the assumption, we deduce from \eqref{eq:Omdec} and \eqref{eq:Shelldec} that 
\begin{equation}\label{eq:measures}
 |\Om \cap B_\beta^c|
 =
 |\Om^c \cap (B_\beta \setminus \overline{B_\alpha})|.
\end{equation} 

Using \eqref{eq:Omdec} and \eqref{eq:Shelldec}, we write
\begin{align}
\notag
\int_{\Omega} H(|x|) \, dx 
&= \int_{\Omega\cap B_\beta}  H(|x|)\, dx 
+
\int_{\Omega \cap B_\beta^c} H(|x|) \, dx\\
\label{eq:H1}
&= 
\int_{B_\beta \setminus \overline{B_\alpha}}  H(|x|)\, dx 
- 
\int_{\Om^c \cap (B_\beta \setminus \overline{B_\alpha})}  H(|x|)\, dx
+
\int_{\Omega \cap B_\beta^c} H(|x|) \, dx
\end{align}
and, in the same manner, 
\begin{equation}\label{eq:G1}
\int_{\Omega} G^2(|x|) \, dx
=  
\int_{B_\beta \setminus \overline{B_\alpha}}  G^2(|x|)\, dx 
- 
\int_{\Om^c \cap (B_\beta \setminus \overline{B_\alpha})}  G^2(|x|)\, dx
+
\int_{\Omega \cap B_\beta^c} G^2(|x|) \, dx.
\end{equation}
Substituting \eqref{eq:H1} and \eqref{eq:G1} into \eqref{eq:mainineq1}, we get
\begin{align}
\notag
&\int_{B_\beta \setminus \overline{B_\alpha}}  H(|x|)\, dx -\int_{\Om^c \cap (B_\beta \setminus \overline{B_\alpha})} H(|x|)\, dx 
+
\int_{\Omega \cap B_\beta^c} H(|x|) \, dx
+
\widetilde{G}
\\
\label{eq:weinbx1}
&\leq
\tau_{l,1} \int_{B_\beta \setminus \overline{B_\alpha}}  G^2(|x|)\, dx 
-
\tau_{l,1}\int_{\Om^c \cap (B_\beta \setminus \overline{B_\alpha})}  G^2(|x|)\, dx
+
\tau_{l,1} \int_{\Omega \cap B_\beta^c} G^2(|x|) \, dx.
\end{align}
By the choice of $v$ and the formula \eqref{eq:taudef-first}, we know that 
\begin{align*}
	\tau_{l,1}
	=
	\frac{\int_\al^\be H(r) r^{N-1} \,dr + h \alpha^{N-1} G^2(\alpha)}{\int_\al^\be G^2(r)r^{N-1} \, dr}
	=
	\frac{\int_{B_\beta \setminus B_\alpha} H(|x|) \,dx + \widetilde{G}}{\int_{B_\beta \setminus B_\alpha} G^2(|x|) \, dx},
\end{align*}
and hence \eqref{eq:weinbx1} simplifies to
\begin{align}
\notag
-\int_{\Om^c \cap (B_\beta \setminus \overline{B_\alpha})} H(|x|)\, dx 
&+
\int_{\Omega \cap B_\beta^c} H(|x|) \, dx
\\
\label{eq:weinbxeq1}
&\leq
-\tau_{l,1} 
\int_{\Om^c \cap (B_\beta \setminus \overline{B_\alpha})}  G^2(|x|)\, dx
+
\tau_{l,1} \int_{\Omega \cap B_\beta^c} G^2(|x|) \, dx.
\end{align}

Observe that for any $x \in \Omega \cap B_\beta^c$ we have $|x| \geq \beta$, and hence
\begin{equation*}\label{eq:eqGH}
G(|x|)=G(\beta)
\quad \text{and} \quad
H(|x|)
=
\frac{l(l+N-2)G^2(\beta)}{|x|^2}
\leq
\frac{l(l+N-2)G^2(\beta)}{\beta^2}=H(\beta),
\end{equation*}
where the inequality for $H$ is strict if and only if $|x| > \beta$ and $l \geq 1$.
Therefore, in view of \eqref{eq:measures}, we obtain
\begin{align}
\label{eq:Hstrict}
\int_{\Omega \cap B_\beta^c} H(|x|) \, dx
&\leq
\int_{\Omega \cap B_\beta^c} H(\beta) \, dx
=
\int_{\Om^c \cap (B_\beta \setminus \overline{B_\alpha})}
H(\beta) \,dx,\\
\label{eq:Gstrict}
\int_{\Omega \cap B_\beta^c} G^2(|x|) \, dx
&=
\int_{\Omega \cap B_\beta^c} G^2(\beta) \, dx
=
\int_{\Om^c \cap (B_\beta \setminus \overline{B_\alpha})} G^2(\beta) \, dx,
\end{align}
and the inequality \eqref{eq:Hstrict} is strict if and only if $|\Omega \cap B_\beta^c| > 0$ and $l \geq 1$.

Thanks to \eqref{eq:Hstrict} and \eqref{eq:Gstrict}, we see that  \eqref{eq:weinbxeq1} (and hence the desired inequality \eqref{eq:bound1}) is satisfied if
\begin{equation}\label{eq:mainineq3}
\int_{\Om^c \cap (B_\beta \setminus \overline{B_\alpha})}
\left[
(H(\beta) - H(|x|))
-
\tau_{l,1}
(G^2(\beta)-G^2(|x|))
\right] dx
\leq 
0.
\end{equation}
Noting that $G(|x|) = v(|x|)$ for any $x \in \Om^c \cap (B_\beta \setminus \overline{B_\alpha})$, we rewrite the integrand from \eqref{eq:mainineq3} as follows:
\begin{align}
\notag
&(H(\beta) - H(|x|))
-
\tau_{l,1}
(G^2(\beta)-G^2(|x|))
\\
\label{eq:weinbfin}
&=
\left(\frac{l(l+N-2)}{\beta^2}-\tau_{l,1}\right)v^2(\beta)
-
\left(\frac{l(l+N-2)}{r^2}-\tau_{l,1}\right)v^2(r)
-
((v'(r))^2.
\end{align}
Finally, Lemma~\ref{lem:v} gives the nonpositivity of \eqref{eq:weinbfin}, and hence \eqref{eq:bound1} is established.
Moreover, thanks to \eqref{eq:measures}, the inequality \eqref{eq:mainineq3} is strict if and only if the domain of integration has positive measure and either $l \geq 1$ or $h \neq 0$.
(Here we notice that if $|\Om \cap B_\beta^c|>0$, $l=0$, and $h \neq 0$, then the inequality \eqref{eq:mainineq3} is strict since $|v'|>0$ in $(\alpha,\beta)$, see Lemma~\ref{lem:v}~\ref{eq:long:1}.)
Recalling that $\Omega$ is Lipschitz, we conclude that $|\Om \cap B_\beta^c|=0$  if and only if $\Omega = B_\beta \setminus{B_\alpha}$, which completes the proof.
\end{proof}

Thanks to the basisness  of eigenfunctions of the form \eqref{eq:sepvar} in $L^2(B_\beta \setminus \overline{B_\alpha})$, we have 
\begin{equation}\label{eq:tau=tau}
\{\tau_k(B_\beta \setminus \overline{B_\alpha})\}_{k \in \mathbb{N}} = \{\tau_{l,j}\}_{l \in \mathbb{N}_0, j \in \mathbb{N}}.
\end{equation}
%Moreover, an eigenvalue $\tau_k(B_\beta \setminus \overline{B_\alpha})$ of \eqref{eq:D} which equals to an eigenvalue $\tau_{l,j}$ of \eqref{eq:ode}-\eqref{eq:bcM} has multiplicity at least $\Lambda_l$, and if $l=0$, then there is a radial eigenfunction associated with $\tau_k(B_\beta \setminus \overline{B_\alpha})$.
In order to apply Proposition~\ref{prop:bound} in the proofs of Theorems~\ref{thm1-ext} and~\ref{thm1}, we should determine the position of $\tau_{l,1}$ in the spectrum of \eqref{eq:D} in $B_\beta \setminus \overline{B_\alpha}$.
Since the entries of the infinite matrix $\{\tau_{l,j}\}$ are increasing along rows and columns by \eqref{eq:row} and \eqref{eqn:col}, respectively, we see from \eqref{eq:tau=tau} that
\begin{equation}\label{eq:tau1tau2}
\tau_1(B_\beta \setminus \overline{B_\alpha})=\tau_{0,1}
\quad \text{and} \quad
\tau_2(B_\beta \setminus \overline{B_\alpha})=
\min\{\tau_{1,1},\tau_{0,2}\}
\end{equation}
for any fixed $h \in \mathbb{R} \cup \{+\infty\}$. 
Let us observe that the sign of $\tau_{1,1}$ depends on $h$, while $\tau_{0,2} > 0$ for any $h \in \mathbb{R} \cup \{+\infty\}$, see Lemmas~\ref{lem:v}~\ref{eq:long:2}:\ref{eq:long:2:iii} and~\ref{lem:tau2>0}, respectively.

We start with describing the ordering of the first several eigenvalues of the SL problem \eqref{eq:ode}-\eqref{eq:bcM}.
%, which is important for the proof of Theorems~\ref{thm1-ext} and~\ref{thm1}.
\begin{lemma}\label{lem:n+2}
Let $\beta>0$ and $l \in \mathbb{N}$. 
Then 
for any $\alpha \in (0,
\beta)$ 
there exists $\bar{h}_l^* \in \mathbb{R} \cup \{+\infty\}$ such that for any $h \in (-\infty,\bar{h}_l^*]$ we have
\begin{equation}\label{eq:muj1<mu02}
\tau_{0,1}
<
\tau_{1,1}
<
\cdots
<
\tau_{l,1}
<
\tau_{0,2}.
\end{equation}
Moreover, there exists $\alpha_l^* \in (0,\beta)$
such that $\bar{h}_l^* = +\infty$ for any $\alpha \in [\alpha_l^*,\beta)$.

Furthermore, in the case $l=2$, for any $\alpha \in (0,\beta)$ 
there also exists $h_2^*<0$ such that the chain of inequalities \eqref{eq:muj1<mu02} holds for any $h \in (-\infty,\bar{h}_2^*]\cup[h_2^*,+\infty]$, i.e., \begin{equation}\label{eq:mu11<mu21<mu02}
 \tau_{0,1}<\tau_{1,1}<\tau_{2,1}<\tau_{0,2}.
 \end{equation}
\end{lemma}

    We refer to Figure~\ref{fig:thm1}   for a schematic graph of points on the $(\alpha,h)$-plane satisfying the restrictions of Lemma~\ref{lem:n+2} (by taking 
    %$l=2^{\kappa-1}$
    %(i.e
    %$\kappa=1+\log_2l \in \mathbb{N}$
    %and 
    $\bar{h}_{\kappa} = \bar{h}_{l}^*$, ${\alpha}_{\kappa} = {\alpha}_{l}^*$, $h_2 = h_2^*$). 
A proof of Lemma~\ref{lem:n+2}  is placed in Appendix \ref{sec:appendix}.
The chain of inequalities \eqref{eq:mu11<mu21<mu02} is given by \cite[Lemma~2.3]{ABD1} in the Neumann case $h=0$, and the inequality $\tau_{1,1}<\tau_{0,2}$ in the Dirichlet-Neumann case $h=+\infty$ can be obtained from \cite[Theorems~1.5]{AK}. 

\begin{remark}\label{rem:radial}
In general, we do not claim that $\bar{h}_2^* < h_2^*$ for any $\alpha \in (0,\alpha_2^*)$.
But already the inequality $\tau_{1,1}<\tau_{0,2}$ in \eqref{eq:mu11<mu21<mu02} might not be true for arbitrary $h<0$ and $0<\alpha<\beta$, at least in the case $N=2$.
Indeed, let us observe that the general solution of the equation \eqref{eq:ode} in the case $N=2$ is given by
\begin{equation}\label{eq:generalsolution}
	v(r) 
	= 
	c_1 J_{l}\left(\sqrt{\tau} r\right) 
	+ c_2Y_{l}\left(\sqrt{\tau} r\right),
\end{equation}
where $J_l$ and $Y_l$ are the $l$-th order Bessel functions  of the first and second kind, respectively. 
The constants $c_1$ and $c_2$ are determined through the imposed boundary conditions \eqref{eq:bcM}:
$$
\left\{
\begin{aligned}
	c_1 
	\left(
	\sqrt{\tau} J_l'\left(\sqrt{\tau}\alpha\right) 
	-
	h 
	J_l\left(\sqrt{\tau} \alpha\right) 
	\right)
	+
	c_2
	\left(
	\sqrt{\tau} Y_l'\left(\sqrt{\tau}\alpha\right) 
	-
	h 
	Y_l\left(\sqrt{\tau} \alpha\right) 
	\right)
	=0,\\
	c_1 
	\sqrt{\tau} J'_l\left(\sqrt{\tau} \beta\right) 
	+
	c_2 
	\sqrt{\tau} Y'_l\left(\sqrt{\tau} \beta\right)
	=0.
\end{aligned}
\right.
$$
In this way, any eigenvalue $\tau_{l,k}$ can be characterized as the $k$-th zero of the following cross-product of Bessel functions:
\begin{align*}
	B_l(\tau) 
	= 
	&\sqrt{\tau} Y_l'\left(\sqrt{\tau} \beta\right)
	\left(
	\sqrt{\tau} J_l'\left(\sqrt{\tau}\alpha\right) 
	-
	h 
	J_l\left(\sqrt{\tau} \alpha\right) 
	\right)
	\\
	&-
	\sqrt{\tau} J_l'\left(\sqrt{\tau} \beta\right)
	\left(
	\sqrt{\tau} Y_l'\left(\sqrt{\tau}\alpha\right) 
	-
	h 
	Y_l\left(\sqrt{\tau} \alpha\right) 
	\right).
\end{align*}
Since the functions $J_l$, $J_l'$, $Y_l$, $Y_l'$ are non-oscillatory, we can find roots of $B_l$ with arbitrary precision using standard numerical methods. 
In particular, numerical investigation shows that if $N=2$, $\alpha=1$, $\beta=15$, and $h=-0.8$, then $\tau_{1,1} \approx 0.0126485$ and $\tau_{0,2} \approx 0.0100829$, that is, $\tau_{0,2} < \tau_{1,1}$, see Figure~\ref{fig:1}.
This interesting fact indicates that the second eigenfunction of the problem \eqref{eq:D} in $B_\beta \setminus \overline{B_\alpha}$ can be radial for certain values of $h<0$, and we believe that it deserves further elaboration. 

\begin{figure}[ht!]
	\centering
	\includegraphics[width=0.75\linewidth]{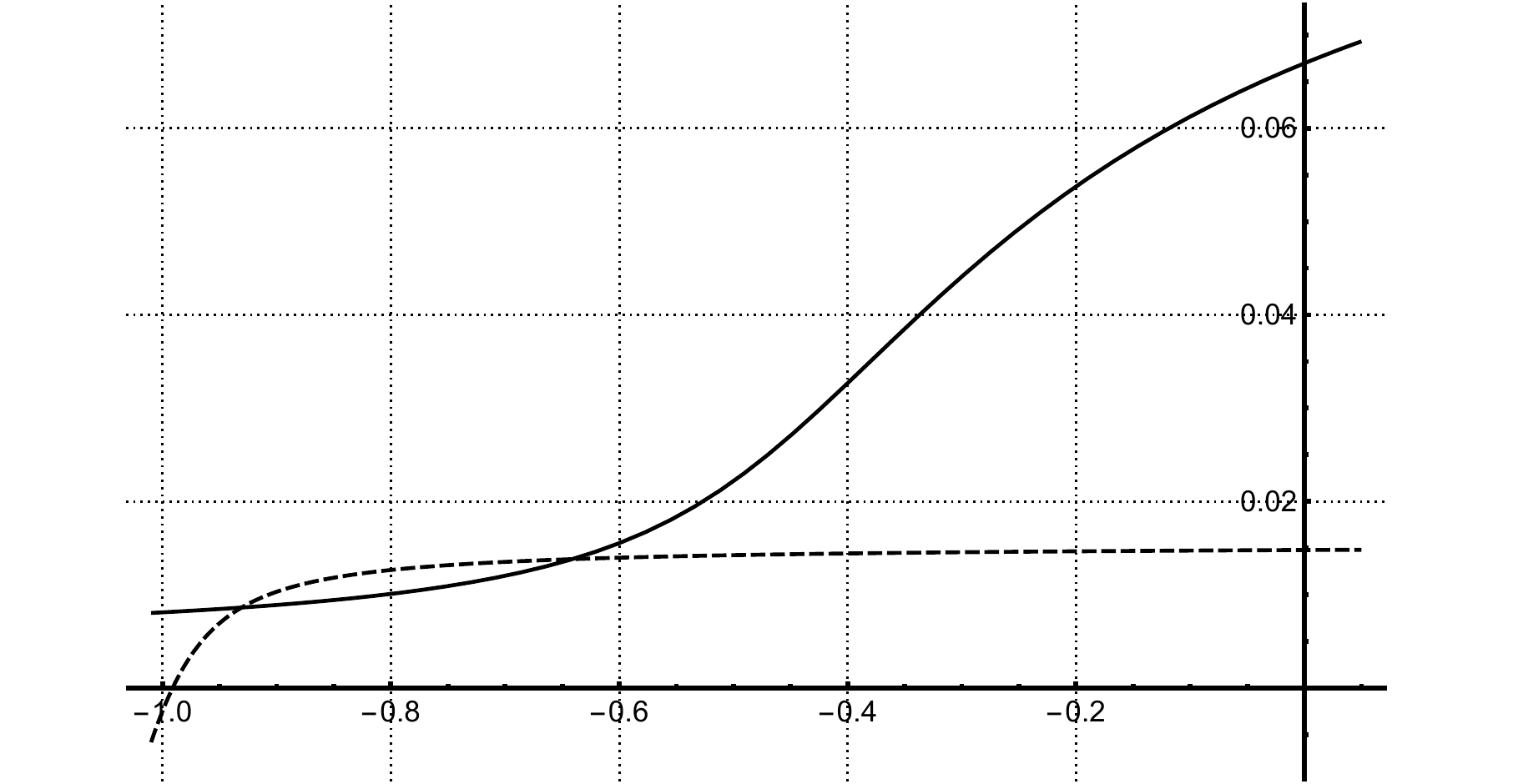}\\
	\caption{The dependence of $\tau_{0,2}$ (thick line) and $\tau_{1,1}$ (dashed line) on $h \in [-1.01,0.5]$, when $N=2$, $\alpha=1$, $\beta=15$.}
	\label{fig:1}
\end{figure}
\end{remark}

\begin{remark}
The inequality \eqref{eq:muj1<mu02}  implies that the index of an eigenvalue of \eqref{eq:D} in $B_\beta \setminus \overline{B_\alpha}$ corresponding to the second radial eigenfunction exceeds any predetermined value if either the spherical shell $B_\beta \setminus \overline{B_\alpha}$ is thin enough, or the Robin parameter $h$ is sufficiently negative.
Notice, however, that \eqref{eq:muj1<mu02} with $l \geq 3$ is not generally true  when $h \in [h_l^*,+\infty]$ with some $h_l^*<0$ since it fails for $h=0$ as $\alpha \to 0$, see the discussion after \cite[Lemma~2.3]{ABD1}.
\end{remark}

Let us now observe that an eigenvalue $\tau_k(B_\beta \setminus \overline{B_\alpha})$ of \eqref{eq:D} which equals to an eigenvalue $\tau_{l,j}$ of \eqref{eq:ode}-\eqref{eq:bcM} has the multiplicity at least $\Lambda_l$ (see~\eqref{eq:Lambdal}).
In view of this fact, the following result, which describes the position of $\tau_{l,1}$ in the spectrum of \eqref{eq:D} in $B_\beta \setminus \overline{B_\alpha}$ and refines~\eqref{eq:tau1tau2},
is a consequence of Lemma~\ref{lem:n+2}, cf.\ \cite[Corollary~2.4]{ABD1} in the Neumann case $h=0$.
\begin{corollary}\label{cor:multiplicity}
Let $0<\alpha<\beta$ and $l \in \mathbb{N}$. 
Let $\bar{h}_l^*$ and $h_2^*$ be as in Lemma~\ref{lem:n+2}.
Then the following assertions are satisfied:
\begin{enumerate}[label={\rm(\roman*)}]
\item\label{cor:multiplicity:1}
If $h \in (-\infty,\bar{h}_2^*]\cup[h_2^*,+\infty]$, then 
\begin{align}
\label{eq:multip3}
\tau_2(B_\beta \setminus \overline{B_\alpha})
&=
\dots
=
\tau_{N+1}(B_\beta \setminus \overline{B_\alpha})
\equiv
\tau_{1+\Lambda_1}(B_\beta \setminus \overline{B_\alpha})
= 
\tau_{1,1},\\
\label{eq:multip4}
\tau_{N+2}(B_\beta \setminus \overline{B_\alpha})
&=
\dots
=
\tau_{\frac{N(N+3)}{2}}(B_\beta \setminus \overline{B_\alpha}) 
\equiv
\tau_{1+\Lambda_1+\Lambda_2}(B_\beta \setminus \overline{B_\alpha}) 
=
\tau_{2,1}.
\end{align}
\item\label{cor:multiplicity:3}
If $h \in (-\infty,\bar{h}_l^*]$, then
\begin{equation}\label{eq:multip5}
\tau_{1+\Lambda_1+\dots+\Lambda_{l-1}+1}(B_\beta \setminus \overline{B_\alpha})
=
\dots
=
\tau_{1+\Lambda_1+\dots+\Lambda_{l-1}+\Lambda_{l}}(B_\beta \setminus \overline{B_\alpha})  =
\tau_{l,1},
\end{equation} 
where in the case $l=1$, \eqref{eq:multip5} is understood as \eqref{eq:multip3}.
In particular, if $N=2$, then  
$$
\tau_{2l}(B_\beta \setminus \overline{B_\alpha})
=
\tau_{2l+1}(B_\beta \setminus \overline{B_\alpha})
=
\tau_{l,1}.
$$
\end{enumerate}
\end{corollary}

We refer to Lemmas~\ref{lem:vneq0}, \ref{lem:tau-derivative}, \ref{lem:compareuv}, \ref{lem:tau} in Appendix~\ref{sec:appendix} for some additional properties of eigenvalues and eigenfunctions of the SL problem \eqref{eq:ode}-\eqref{eq:bcM} needed for the proofs of Lemmas~\ref{lem:v} and~\ref{lem:n+2}. 
See also \cite{EMZ,KZ,KWZ2} for further results in this direction.

\section{Proofs of the main results}\label{sec:proof2}

\subsection{Proof of Theorem~\ref{thm0}}
Let $v$ be a positive  eigenfunction corresponding to the eigenvalue $\tau_{0,1}$ of the SL problem \eqref{eq:ode}-\eqref{eq:bcM} with $l=0$ and $h \neq 0$.
As in Proposition~\ref{prop:bound}, we define the function
\begin{equation*}
	G(r)
	=
	\left\{
	\begin{aligned}
		&v(r) 	&&\text{for}~ r \in [\alpha,\beta), \\
		&v(\beta) &&\text{for}~ r \geq \beta,
	\end{aligned}
	\right.
\end{equation*}
and observe that $G(r)$ is a $C^1$-function in $[\alpha,+\infty)$.
Hence, we can use the function $x \mapsto G(|x|)$, $x \in \Omega$, as a trial function for the definition \eqref{eq:tau1} of $\tau_1(\Omega)$. 
By Proposition~\ref{prop:bound} with $l=0$ and the first equality in \eqref{eq:tau1tau2}, we obtain the desired bound
\begin{align}
	\notag
	\tau_1(\Omega)
	&\leq
	\frac{\int_{\Omega} |\nabla G(|x|)|^2 \,dx + h \int_{\partial B_\alpha} G^2(|x|) \,dS}{\int_{\Omega} G^2(|x|) \,dx}
	\\
	\label{eq:tau1estim}
	&=
	\frac{\int_{\Omega} (G'(|x|))^2 \,dx + h \int_{\partial B_\alpha} G^2(|x|) \,dS}{\int_{\Omega} G^2(|x|) \,dx}
	\leq 
	\tau_{0,1} = \tau_1(B_\beta \setminus \overline{B_\alpha}).
\end{align}
Moreover, Proposition~\ref{prop:bound} also gives the strict inequality in \eqref{eq:tau1estim} provided $\Omega \neq B_\beta \setminus \overline{B_\alpha}$, since $h \neq 0$.
\qed

\medskip
Let us turn to the proofs of 
Theorems~\ref{thm1-ext} and \ref{thm1}.
We start with a few preliminary remarks.
It is known from the general theory that the first eigenvalue $\tau_{1}(\Omega)$ is simple and the corresponding first eigenfunction $\phi_1$ has a constant sign in $\Omega$.
We assume, without loss of generality, that $\phi_1>0$ in $\Omega$.
It is not hard to observe that $\phi_1$ inherits symmetries of $\Omega$.
Indeed, in the opposite case, the composition of $\phi_1$ with an appropriate element of the symmetry group would be another first eigenfunction linearly independent from $\phi_1$, which contradicts the simplicity of $\tau_{1}(\Omega)$.

In order to prove Theorems~\ref{thm1-ext} and \ref{thm1}, it will be  convenient to work with a slightly different minimax characterization of $\tau_{k}(\Omega)$, namely,
\begin{equation}\label{eq:taukdef}
\tau_{k}(\Omega) 
= 
\min_{Y\in\mathcal{Y}_{k-1}} 
\max_{u \in Y\setminus \{0\}} \frac{\int_{\Omega} |\nabla u|^2 \, dx + h \int_{\partial B_\alpha} u^2 \,dS}{\int_{\Omega} u^2 \, dx},
\quad
k \geq 2,
\end{equation} 
where $\mathcal{Y}_{k-1}$ is the collection of all $(k-1)$-dimensional subspaces of $\tilde{H}^1$ which are $L^2(\Omega)$-orthogonal to $\mathbb{R} \phi_1$. 
For the sake of clarity,  we justify in Lemma~\ref{lem:tau=tautilde} that the characterizations \eqref{eq:taukdef} and \eqref{eq:tauk} are equivalent.

\subsection{Proof of Theorem~\ref{thm1-ext}}
The following important observation partially motivates our symmetry assumption on $\Omega$ and will be used in the proof: since $\Omega$ is assumed to be symmetric of order $2^{\kappa}$ with $\kappa \in \mathbb{N}$, it is  symmetric of order $2^{t+1}$ for any $t \in \{0,\dots,\kappa-1\}$.
Hence, in the polar coordinates $(r,\theta)$, both the domain $\Omega$ and the corresponding first eigenfunction $\phi_1$ of \eqref{eq:D} are invariant under the mapping $\theta \mapsto \theta + \pi/2^{t}$ for any $t \in \{0,\dots,\kappa-1\}$.

We start by proving the inequality \eqref{eq:SWhigh} for the largest admissible index $i=2^{\kappa}$.
Since the multiplicity $\Lambda_j=2$ for any $j \in \mathbb{N}$ in the planar case $N=2$ (see \eqref{eq:Lambdal}), we deduce from Corollary~\ref{cor:multiplicity} (by taking $l=2^{\kappa-1}$ and denoting $\bar{h}_{\kappa} = \bar{h}_{l}^*$, ${\alpha}_{\kappa} = {\alpha}_{l}^*$, $h_2 = h_2^*$) that, under the corresponding assumptions on $h$ and $\alpha$, the desired inequality \eqref{eq:SWhigh} will follow from the inequality
\begin{equation}\label{eq:inequtautau1}
	\tau_{2^{\kappa}}(\Omega) \leq \tau_{2^{\kappa-1},1}.
\end{equation}
Let us first assume that $\kappa  \geq 2$.
In order to prove \eqref{eq:inequtautau1}, we consider a subspace ${Y}_{2^{\kappa}-1} \subset \tilde{H}^1(\Omega)$ defined in the polar coordinates $(r,\theta)$ as 
\begin{align*}
	{Y}_{2^{\kappa}-1} 
	= 
	\text{span}
	\bigg\{ 
	&G(r)\sin(\theta),
	G(r)\cos (\theta),
	\\
	&G(r)\sin(2\theta),
	G(r)\cos (2\theta),
	\\
	&G(r)\sin(3\theta),
	G(r)\cos (3\theta),
	\dots,
	\\
	&G(r)\sin ((2^{\kappa-1}-1)\theta),
	G(r)\cos ((2^{\kappa-1}-1)\theta), w
	\bigg\},
\end{align*}
where we assume that
\begin{equation}\label{eq:wfunc}
	\text{either} \quad 
	w = G(r) 
	\sin (2^{\kappa-1}\theta)
	\quad \text{or} \quad 
	w = G(r) \cos(2^{\kappa-1}\theta),
\end{equation}
and its precise choice will be nonconstructively defined later (see Step~6).
The function 
$G$ is a constant extension of a positive eigenfunction $v$ corresponding to $\tau_{2^{\kappa-1},1}$ as in Proposition~\ref{prop:bound}.
Our aim is to show that ${Y}_{2^{\kappa}-1}$ is an admissible subspace for the definition \eqref{eq:taukdef} of 
$\tau_{2^{\kappa}}(\Omega)$ and the maximum of the corresponding Rayleigh quotient over ${Y}_{2^{\kappa}-1}$ equals $\tau_{2^{\kappa-1},1}$.
We split the consideration into several steps.

\textbf{Step 1. $L^2(\Omega)$-orthogonality to $\phi_1$.}
Taking any $i \in \{1,2,3,\dots,2^{\kappa-1}\}$\footnote{Within the proof, ``$i$'' always stands for a natural number.}, let us show that
$$
\int_\Omega 
G(r)\sin (i\theta)
\phi_1(r,\theta) r \, dr d\theta=0
\quad \text{and} \quad
\int_\Omega 
G(r)\cos (i\theta)
\phi_1(r,\theta) r \, dr d\theta=0.
$$
As a consequence of the prime factorization, we can decompose $i = 2^{t} R$, where $t \in \{0,\dots,\kappa-1\}$ and $R$ is odd. 
Consequently,  
$\Omega$ and $\phi_1$ are invariant under the mapping $\theta \mapsto \theta + \pi/2^{t}$, $i/2^{t}$ is odd, and hence 
\begin{equation}
	\label{eq:sincossincson}
	\sin \left(i\theta+i\frac{\pi}{2^t}\right)
	=
	-
	\sin (i\theta)
	\quad \text{and} \quad 
	\cos \left(i\theta+i\frac{\pi}{2^t}\right)
	=
	-
	\cos (i\theta).
	%\quad x \in \mathbb{R}.
\end{equation}
Now making the change of variables $\theta \mapsto \theta + \pi/2^{t}$,
we get
\begin{align*}
	&\int_\Omega 
	G(r)\sin (i\theta)
	\phi_1(r,\theta) r \, dr d\theta
	\\
	&=
	\int_\Omega 
	G(r)\sin \left(i\left(\theta+\frac{\pi}{2^t}\right)\right)
	\phi_1\left(r,\theta+\frac{\pi}{2^t}\right) r \, dr d\theta
	\\
	&=
	-\int_\Omega 
	G(r)\sin (i\theta)
	\phi_1(r,\theta) r \, dr d\theta,
\end{align*}
and the same is true with $\cos$ instead of $\sin$.
Consequently, each basis element of $Y_{2^{\kappa}-1}$ is $L^2(\Omega)$-orthogonal to $\phi_1$, and hence so is any other element of $Y_{2^{\kappa}-1}$.

\textbf{Step 2. Mutual $L^2(\Omega)$-orthogonality.}
Let us show that basis elements of $Y_{2^{\kappa}-1}$ are $L^2(\Omega)$-orthogonal to each other. 
Since $\cos(x) = \sin(x+\pi/2)$, it is sufficient to prove that \begin{equation}\label{eq:orth1}
	\int_\Omega 
	G^2(r)
	\sin \left(i\theta+n\frac{\pi}{2}\right)
	\sin \left(j\theta+m\frac{\pi}{2}\right)
	r \, dr d\theta
	=
	0
\end{equation}
for any $n,m \in \{0,1\}$ and $i,j \in \{1,\dots,2^{\kappa-1}\}$ with $i \geq j$ such that if $i =j$, then $m \neq n$, and if $i=2^{\kappa-1}$, then $i>j$.
We use the standard identity
\begin{align}
	\notag
	\sin \left(i\theta+n\frac{\pi}{2}\right)
	&\sin \left(j\theta+m\frac{\pi}{2}\right)
	\\
	\label{eq:sinsincos1}
	&=
	\frac{1}{2}
	\cos \left((i-j)\theta + \frac{(n-m)\pi}{2}\right)
	-
	\frac{1}{2}
	\cos \left((i+j)\theta + \frac{(n+m)\pi}{2}\right).
\end{align}
First, we deal with the second term on the right-hand side of \eqref{eq:sinsincos1}.
We have 
$i+j \in \{2,\dots,2^{\kappa-1}+(2^{\kappa-1}-1)\}$, i.e.,
$i+j \in \{2,\dots,2^{\kappa}-1\}$.
As above, we can write $i+j = 2^t R$, where $t \in \mathbb{N}_0$ and $R$ is odd.
It is clear that $t \in \{0,\dots,\kappa-1\}$ since $2^t \leq 2^{\kappa}-1$.
Therefore, $(i+j)/2^t$ is odd, and $\Omega$ is symmetric of order $2^{t+1}$. 
Consequently, using \eqref{eq:sincossincson}, we get
\begin{align*}
	&\int_\Omega 
	G(r)\cos \left((i+j)\theta+\frac{(n+m)\pi}{2}\right) r \, dr d\theta
	\\
	&=
	\int_\Omega 
	G(r)\cos \left((i+j)\left(\theta+\frac{\pi}{2^t}\right)+\frac{(n+m)\pi}{2}\right)
	r \, dr d\theta
	\\
	&=
	-\int_\Omega 
	G(r)\cos \left((i+j)\theta+\frac{(n+m)\pi}{2}\right)
	r \, dr d\theta,
\end{align*}
which yields 
$$
\int_\Omega 
G(r)\cos \left((i+j)\theta+\frac{(n+m)\pi}{2}\right)
r \, dr d\theta=0.
$$
In much the same way, we deal with the first term on the right-hand side of \eqref{eq:sinsincos1}.
Namely, if $i > j$, then we get
$$
\int_\Omega 
G(r)\cos \left((i-j)\theta+\frac{(n-m)\pi}{2}\right)
r \, dr d\theta=0.
$$
On the other hand, if $i=j$, then $m \neq n$, and we have
$$
\int_\Omega 
G(r)\cos \left((i-j)\theta+\frac{(n-m)\pi}{2}\right)
r \, dr d\theta
=
\int_\Omega 
G(r)\cos \left(\frac{\pi}{2}\right)
r \, dr d\theta = 0.
$$
Combining the last three displayed equations, we conclude that \eqref{eq:orth1} is satisfied, i.e., basis elements of $Y_{2^{\kappa}-1}$ are $L^2(\Omega)$-orthogonal to each other.

\textbf{Step 3. Dimension.} 
%$\text{dim}\,{Y}_{2^{\kappa}-1} = 2^{\kappa}-1$.
Plainly, ${Y}_{2^{\kappa}-1}$ has $2^{\kappa}-1$ basis elements.
The mutual $L^2(\Omega)$-orthogonality from Step~2 implies that basis elements are linearly independent, that is,
$\text{dim}\,{Y}_{2^{\kappa}-1} = 2^{\kappa}-1$.

It follow from Steps 1 and 3 that that ${Y}_{2^{\kappa}-1}$ is an admissible subspace for the definition \eqref{eq:taukdef} of 
$\tau_{2^{\kappa}}(\Omega)$. 
In the subsequent Steps 4, 5, 6, we prepare auxiliary properties of ${Y}_{2^{\kappa}-1}$ in order to show, in Step 7, that ${Y}_{2^{\kappa}-1}$ delivers the upper bound $\tau_{2^{\kappa-1},1}$ for $\tau_{2^{\kappa}}(\Omega)$.

\textbf{Step 4. Mutual $H^1_0(\Omega)$-orthogonality.}
Let us show that basis elements of $Y_{2^{\kappa}-1}$ are orthogonal to each other in the $H^1_0(\Omega)$-norm (i.e., their gradients are $L^2(\Omega)$-orthogonal). 
For brevity, consider a function $v_{i,n}=v_{i,n}(x)$ defined as
$$
v_{i,n}(x) 
= 
G(r(x)) \sin\left(i\theta(x) + n\frac{\pi}{2}\right).
$$
We have
\begin{align*}
	\frac{\partial v_{i,n}(x)}{\partial x_1}
	&=
	G'(r) \sin\left(i\theta + n\frac{\pi}{2}\right) \cos(\theta)
	-
	i \frac{G(r)}{r}
	\cos\left(i\theta + n\frac{\pi}{2}\right) \sin(\theta),\\
	\frac{\partial v_{i,n}(x)}{\partial x_2}
	&=
	G'(r) \sin\left(i\theta + n\frac{\pi}{2}\right) \sin(\theta)
	+
	i \frac{G(r)}{r}
	\cos\left(i\theta + n\frac{\pi}{2}\right) \cos(\theta).
\end{align*}
Straightforward calculations give 
\begin{align}
	\notag
	\intO \langle \nabla v_{i,n}, \nabla v_{j,m}\rangle_{\mathbb{R}^2} \,dx
	&=
	\intO (G')^2(r) 
	\sin \left(i\theta+n\frac{\pi}{2}\right)
	\sin \left(j\theta+m\frac{\pi}{2}\right)
	r \, dr d\theta
	\\
	\label{eq:nablavxy1}
	&+
	ij\intO \frac{G^2(r)}{r^2}
	\cos \left(i\theta+n\frac{\pi}{2}\right)
	\cos \left(j\theta+m\frac{\pi}{2}\right)
	r \,dr d\theta.
\end{align}
Let us now fix any $n,m \in \{0,1\}$ and $i,j \in \{1,\dots,2^{\kappa-1}\}$ with $i \geq j$ such that if $i =j$, then $m \neq n$, and if $i=2^{\kappa-1}$, then $i>j$.
Using exactly the same arguments as in Step 2, we see that each integral on the right-hand side of \eqref{eq:nablavxy1} is zero.
This gives the desired orthogonality in the $H^1_0(\Omega)$-norm.

\textbf{Step 5. Norms I.}
Take any $i \in \{1,\dots,2^{\kappa-1}-1\}$ and denote 
\begin{equation}\label{eq:ai-0}
	A_i 
	=
	\int_\Omega 
	G^2(r)\sin^2 \left(i\theta\right)
	r \, dr d\theta.
\end{equation}
We start by showing that
\begin{equation}\label{eq:ai-1}
	A_i
	=
	\int_\Omega 
	G^2(r)\cos^2 (i\theta)
	r \, dr d\theta
	\equiv
	\int_\Omega 
	G^2(r)\sin^2 \left(i\theta+\frac{\pi}{2}\right)
	r \, dr d\theta.
\end{equation}
Since $i \leq 2^{\kappa-1}-1$, we can decompose $i=2^tR$, where $t \in \{0,\dots,\kappa-2\}$ and $R$ is odd. 
Therefore, $\Omega$ is symmetric of order $2^{t+2}$ and
$$
\sin^2 \left(i\theta+\frac{i\pi}{2^{t+1}}\right) 
= 
\sin^2 \left(i\theta+\frac{R\pi}{2}\right)
=
\cos^2 \left(i\theta\right).
$$
Hence, making the change of variables $\theta \mapsto \theta + i\pi/2^{t+1}$,
we obtain \eqref{eq:ai-1} as follows:
\begin{align*}
	A_i 
	=
	\int_\Omega 
	G^2(r)\sin^2 \left(i\theta + \frac{i\pi}{2^{t+1}}\right)
	r \, dr d\theta
	&=
	\int_\Omega 
	G^2(r)\cos^2 \left(i\theta\right)
	r \, dr d\theta.
\end{align*}
Summing now \eqref{eq:ai-0} and \eqref{eq:ai-1}, we find that
\begin{equation}\label{eq:ai-main}
	A_i 
	= 
	\frac{1}{2}\int_\Omega 
	G^2(r) r \, dr d\theta
	\equiv
	\frac{1}{2}\int_\Omega 
	G^2(|x|)\, dx.
\end{equation}

Arguing in much the same way as above and using the expression \eqref{eq:nablavxy1}, we take any $i \in \{1,\dots,2^{\kappa-1}-1\}$, denote 
\begin{equation}\label{eq:bi-0}
	B_i 
	=
	\int_\Omega 
	\left|\nabla_{x} \left(G(r)\sin (i\theta)\right)\right|^2
	r \, dr d\theta
\end{equation}
and deduce that
\begin{align}
	\label{eq:Bisin}
	B_i
	&=
	\intO \left((G')^2(r) \sin^2(i\theta) + \frac{i^2 G^2(r)}{r^2} \cos^2(i\theta)\right) r \,drd\theta 
	\\
	\notag
	&=
	\intO \left((G')^2(r) \sin^2\left(i\theta + \frac{i\pi}{2^{t+1}}\right) + \frac{i^2 G^2(r)}{r^2} \cos^2\left(i\theta + \frac{i\pi}{2^{t+1}}\right)\right) r \,drd\theta
	\\
	\label{eq:bi-1}
	&=
	\intO \left((G')^2(r) \cos^2(i\theta) + \frac{i^2 G^2(r)}{r^2} \sin^2(i\theta)\right) r \,drd\theta 
	\\
	&=
	\label{eq:Bicos}
	\int_\Omega 
	\left|\nabla_{x} \left(G(r)\cos (i\theta)\right)\right|^2
	r \, dr d\theta.
\end{align}
Moreover, summing \eqref{eq:Bisin} and \eqref{eq:bi-1}, we get
\begin{equation}\label{eq:bi-main}
	B_i = \frac{1}{2}
	\int_\Omega 
	\left((G')^2(r) + \frac{i^2 G^2(r)}{r^2} \right)r \, dr d\theta
	\equiv
	\frac{1}{2}\int_\Omega 
	\left((G')^2(|x|) + \frac{i^2 G^2(|x|)}{|x|^2} \right) dx.
\end{equation}

Analogously, 
denoting
$$
C_i = 
h \int_{\partial B_\alpha}
G^2(|x|)\sin^2 (i\theta(x))\,dS,
$$
we see that
$$
C_i = 
h \int_{\partial B_\alpha}
G^2(|x|)\cos^2 (i\theta(x))\,dS,
$$
and hence
\begin{equation}\label{eq:ci-main}
	C_i = 
	\frac{h}{2} \int_{\partial B_\alpha}
	G^2(|x|)\,dS.
\end{equation}
Notice that \eqref{eq:ci-main} is valid for any $i \in \mathbb{N}$ since $\partial B_\alpha$ is symmetric of any order $q \in \mathbb{N}$.

We conclude from \eqref{eq:ai-main}, \eqref{eq:bi-main}, and \eqref{eq:ci-main} that
$$
\frac{B_i+C_i}{A_i}
=
\frac{\int_\Omega 
	\left((G')^2(|x|) + \frac{i^2 G^2(|x|)}{|x|^2} \right) dx
	+
	h \int_{\partial B_\alpha}
	G^2(|x|)\,dS
}{\int_\Omega 
	G^2(|x|)\, dx}.
$$
Clearly, 
\begin{equation}\label{eq:biciai}
	\frac{B_i+C_i}{A_i} \leq \frac{B_j+C_j}{A_j}
	\quad \text{provided}~ 1 \leq i \leq j \leq 2^{\kappa-1}-1.
\end{equation}

\textbf{Step 6. Norms II.}
In the previous step, we considered the norms of all basis elements of $Y_{2^{\kappa}-1}$ except the very last one -- $w$.
It is clear from the above arguments that the symmetry of order $2^{\kappa}$ is not enough to prove the equality of \eqref{eq:ai-0} to \eqref{eq:ai-1} in the case $i=2^{\kappa-1}$. 
Because of that, we proceed differently.
First, we denote
\begin{gather}
	\tilde{A}_1 
	=
	\int_\Omega 
	G^2(r)\sin^2 (2^{\kappa-1}\theta)
	r \, dr d\theta,
	\quad
	\tilde{A}_2 
	=
	\int_\Omega 
	G^2(r)\cos^2 (2^{\kappa-1}\theta)
	r \, dr d\theta, \qquad
	\\
	\label{eq:b0b0}
	\tilde{B}_1 
	=
	\int_\Omega 	\left|\nabla_{x} \left(G(r)\sin(2^{\kappa-1}\theta)\right)\right|^2
	r \, dr d\theta,
	\quad
	\tilde{B}_2
	=
	\int_\Omega 	\left|\nabla_{x} \left(G(r)\cos (2^{\kappa-1}\theta)\right)\right|^2
	r \, dr d\theta, \qquad
	\\
	\label{eq:c0c0}
	\tilde{C}_1 = 
	h \int_{\partial B_\alpha}
	G^2(|x|)\sin^2 (2^{\kappa-1}\theta(x))\,dS,
	\quad 
	\tilde{C}_2 = 
	h \int_{\partial B_\alpha}
	G^2(|x|)\cos^2 (2^{\kappa-1}\theta(x))\,dS. \qquad
\end{gather}
We deduce from 
\eqref{eq:nablavxy1}
(more precisely, see the equalities \eqref{eq:bi-0}$=$\eqref{eq:Bisin} and \eqref{eq:Bicos}$=$\eqref{eq:bi-1}, which remain valid for $i=2^{\kappa-1}$)
that
\begin{align}
\label{eq:tildeB1}
	\tilde{B}_1 
	&=
	\intO \left((G')^2(r) \sin^2(2^{\kappa-1}\theta) + \frac{4^{\kappa-1} G^2(r)}{r^2} \cos^2(2^{\kappa-1}\theta)\right) r \,drd\theta ,
	\\
 \label{eq:tildeB2}
	\tilde{B}_2 
	&=
	\intO \left((G')^2(r) \cos^2(2^{\kappa-1}\theta) + \frac{4^{\kappa-1} G^2(r)}{r^2} \sin^2(2^{\kappa-1}\theta)\right) r \,drd\theta.
\end{align}
Let us now suppose, by contradiction, that
\begin{equation}\label{eq:B0C0A0}
	\frac{\tilde{B}_1 + \tilde{C}_1}{\tilde{A}_1} 
	>
	\frac{\int_\Omega 
		\left((G')^2(|x|) + \frac{4^{\kappa-1} G^2(|x|)}{|x|^2} \right) dx
		+
		h \int_{\partial B_\alpha}
		G^2(|x|)\,dS
	}{\int_\Omega 
		G^2(|x|)\, dx}
\end{equation}
and
\begin{equation}\label{eq:B1C1A1}
	\frac{\tilde{B}_2 + \tilde{C}_2}{\tilde{A}_2} 
	>
	\frac{\int_\Omega 
		\left((G')^2(|x|) + \frac{4^{\kappa-1} G^2(|x|)}{|x|^2} \right) dx
		+
		h \int_{\partial B_\alpha}
		G^2(|x|)\,dS
	}{\int_\Omega 
		G^2(|x|)\, dx}.
\end{equation}
Multiplying \eqref{eq:B0C0A0} by $\tilde{A}_1$ and \eqref{eq:B1C1A1} by $\tilde{A}_2$,  summing the obtained expressions, 
and noting that $\tilde{A}_1 + \tilde{A}_2 = \int_\Omega G^2(|x|) \,dx$, 
we get
\begin{align}
\notag
	(\tilde{B}_1 + \tilde{C}_1 + \tilde{B}_2 + \tilde{C}_2)
	&>
	\frac{\int_\Omega 
		\left((G')^2(|x|) + \frac{4^{\kappa-1} G^2(|x|)}{|x|^2} \right) dx
		+
		h \int_{\partial B_\alpha}
		G^2(|x|)\,dS
	}{\int_\Omega 
		G^2(|x|)\, dx} (\tilde{A}_1+\tilde{A}_2)
	\\
 \label{eq:BBCC}
	&=
	\int_\Omega 
	\left((G')^2(|x|) + \frac{4^{\kappa-1} G^2(|x|)}{|x|^2} \right) dx
	+
	h \int_{\partial B_\alpha}
	G^2(|x|)\,dS.
\end{align}
On the other hand, by the expressions \eqref{eq:tildeB1}, \eqref{eq:tildeB2} and \eqref{eq:c0c0} for $\tilde{B}_1$, $\tilde{B}_2$ and $\tilde{C}_1$, $\tilde{C}_2$, respectively, 
we have
$$
\tilde{B}_1 + \tilde{C}_1 + \tilde{B}_2 + \tilde{C}_2
=
\intO \left((G')^2(r) + \frac{4^{\kappa-1} G^2(r)}{r^2} \right) r \,drd\theta
+   
h \int_{\partial B_\alpha}
G^2(|x|)\,dS,
$$
which is a contradiction to \eqref{eq:BBCC}.
That is, there exists $j \in \{0,1\}$ such that
\begin{equation}\label{eq:biciai2}
	\frac{\tilde{B}_j + \tilde{C}_j}{\tilde{A}_j} 
	\leq 
	\frac{\int_\Omega 
		\left((G')^2(|x|) + \frac{4^{\kappa-1} G^2(|x|)}{|x|^2} \right) dx
		+
		h \int_{\partial B_\alpha}
		G^2(|x|)\,dS
	}{\int_\Omega 
		G^2(|x|)\, dx}.
\end{equation}
This $j$ defines the exact choice of $w$, see \eqref{eq:wfunc}.

\textbf{Step 7. Upper bound for the Rayleigh quotient.}
Finally, we are ready to show that  ${Y}_{2^{\kappa}-1}$ delivers the upper bound $\tau_{2^{\kappa-1},1}$ for $\tau_{2^{\kappa}}(\Omega)$.
Let us take any $u \in Y_{2^{\kappa}-1} \setminus \{0\}$. 
There exists a nonzero vector $(c_{1,1},c_{1,2},\dots,c_{2^{\kappa-1}-1,1},c_{2^{\kappa-1}-1,2},\tilde{c})$ such that
\begin{align*}
	u 
	= 
	c_{1,1} G(r)\sin(\theta)
	+
	c_{1,2} G(r)\cos(\theta)  
	+
	\dots 
	+
	\tilde{c} \,w.
\end{align*}
Because of the $L^2(\Omega)$- and $H_0^1(\Omega)$-orthogonality and the expressions for the norms 
proved in the above steps, and thanks to \eqref{eq:biciai}, \eqref{eq:biciai2}, we have
\begin{align*}
	\frac{\intO |\nabla u|^2 \,dx}{\intO u^2 \,dx}
	&=
	\frac{(c_{1,1}^2+c_{1,2}^2) (B_1+C_1) + \dots + \tilde{c}^2 (\tilde{B}_j+\tilde{C}_j)}{(c_{1,1}^2+c_{1,2}^2) A_1 + \dots + \tilde{c}^2 \tilde{A}_j}
	\\
	&\leq
	\max\left\{\frac{B_1+C_1}{A_1},\dots,  \frac{\tilde{B}_j+\tilde{C}_j}{\tilde{A}_j}\right\} 
	\\
	&\leq
	\frac{\int_\Omega 
		\left((G')^2(|x|) + \frac{4^{\kappa-1} G^2(|x|)}{|x|^2} \right) dx
		+
		h \int_{\partial B_\alpha}
		G^2(|x|)\,dS
	}{\int_\Omega 
		G^2(|x|)\, dx}.
\end{align*}
Applying now Proposition~\ref{prop:bound} with $N=2$ and $l=2^{\kappa-1}$, we derive the desired inequality \eqref{eq:inequtautau1}:
$$
\tau_{2^{\kappa}}(\Omega) 
\leq 
\max_{u \in Y_{2^{\kappa}-1} \setminus \{0\}} 
\frac{\intO |\nabla u|^2 \,dx + h \int_{\partial B_\alpha} u^2 \,dS}{\intO u^2 \,dx}
\leq \tau_{2^{\kappa-1},1}.
$$

\medskip
Let us now discuss the inequality \eqref{eq:SWhigh} for $i < 2^{\kappa}$ assuming $\kappa \geq 2$.
For any odd index $i \in \{3,\dots,2^{\kappa}-1\}$, we consider the following ``truncation'' of the set $Y_{2^{\kappa}-1}$:
\begin{align*}
	{Y}_{i-1} 
	= 
	\text{span}
	\bigg\{ 
	&G(r)\sin(\theta),
	G(r)\cos (\theta),
	\\
	&G(r)\sin(2\theta),
	G(r)\cos (2\theta),
	\dots,
	\\
	&G(r)\sin ((i-1)\theta/2),
	G(r)\cos ((i-1)\theta/2)
	\bigg\},
\end{align*}
where $G$ is a constant extension of a positive eigenfunction $v$  corresponding to $\tau_{\frac{i-1}{2},1}$.
Arguing exactly in the same way as above (and even simpler, since Step~6 is not needed), we obtain
\begin{align*}
\tau_{i}(\Omega)
&\leq
\max_{u \in Y_{i-1} \setminus \{0\}} 
\frac{\intO |\nabla u|^2 \,dx + h \int_{\partial B_\alpha} u^2 \,dS}{\intO u^2 \,dx}
\\
&\leq
\frac{\int_\Omega 
		\left((G')^2(|x|) + \left(\frac{i-1}{2}\right)^2 \frac{G^2(|x|)}{|x|^2} \right) dx
		+
		h \int_{\partial B_\alpha}
		G^2(|x|)\,dS
	}{\int_\Omega 
		G^2(|x|)\, dx}
  \leq 
  \tau_{\frac{i-1}{2},1}.
\end{align*}
Applying Corollary~\ref{cor:multiplicity}, we get
$$
\tau_{i-1}(\Omega) 
\leq 
\tau_{i}(\Omega) \leq \tau_{\frac{i-1}{2},1} 
= 
\tau_{i-1}(B_\beta \setminus \overline{B_\al}) 
= \tau_i(B_\beta \setminus \overline{B_\al}),
$$
which gives the desired inequality \eqref{eq:SWhigh} for $i < 2^{\kappa}$.

It remains to prove \eqref{eq:SWhigh} for $\kappa = 1$.
In this case, we set $Y_1 = \{w\} \equiv \mathbb{R} w$.
Using Step~1, we see that $Y_1$ is admissible for the definition of $\tau_2(\Omega)$. 
Thanks to Step~6 (see \eqref{eq:biciai2}) and Proposition~\ref{prop:bound} with $N=2$ and $l=1$, we get $\tau_2(\Omega) \leq \tau_{1,1}$, which yields \eqref{eq:SWhigh} in view of Corollary~\ref{cor:multiplicity}~\ref{cor:multiplicity:1}.

Finally, the equality case in \eqref{eq:SWhigh} is covered by Proposition~\ref{prop:bound}.
\qed

\subsection{Proof of Theorem~\ref{thm1}}
The proof of Theorem \ref{thm1} goes along the same lines as the proof of \cite[Theorem~1.6]{ABD1} about the Neumann case $h=0$.
Most of the arguments from \cite[Theorem~1.6]{ABD1} either transfer unchanged to the case of mixed boundary conditions or have corresponding counterparts.
Because of this, we will be sketchy. 

In view of the discussion about $\phi_1$ on page~\pageref{sec:proof2},  \textit{all} the integral identities from  \cite[Appendix~B]{ABD1} remain valid by substituting the standard Lebesgue measure $dx$ with the weighted measure  $\phi_1(x)\,dx$.
Let us also explicitly note that, for domains satisfying the assumption~\ref{assumption}, the positivity of a radial function $g$ in $\mathbb{R}^N$ required in \cite[Appendix~B]{ABD1} can be weakened to its positivity only in $\mathbb{R}^N \setminus \overline{B_\alpha}$ with no changes in the proofs.

\ref{thm1:1}
As a particular case of the minimax characterization \eqref{eq:taukdef2} of $\tau_k(\Omega)$, the second eigenvalue  can be defined as
\begin{equation}\label{eq:tau2def}
\tau_2(\Omega) 
= 
\inf
\left\{
\frac{\int_{\Omega} |\nabla u|^2 \, dx + h \int_{\partial B_\alpha} u^2 \,dS}{\int_{\Omega} u^2 \, dx}:~ 
u \in \tilde{H}^1(\Omega),~ 
\int_{\Omega} \phi_1 u \, dx = 0
\right\}.
\end{equation}
Let $\Omega$ be either symmetric of order $2$ or centrally symmetric.
In view of Corollary~\ref{cor:multiplicity} \ref{cor:multiplicity:1}, it is sufficient to show that $\tau_2(\Omega) \leq \tau_{1,1}$ provided  $h \in (-\infty,\bar{h}_2] \cup [h_2,+\infty]$, where we recall that $\tau_{1,1}$ is the first eigenvalue of the SL problem \eqref{eq:ode}-\eqref{eq:bcM} with $l=1$.

Let $v$ be a positive eigenfunction corresponding $\tau_{1,1}$ and 
define the function $G$ as in Proposition~\ref{prop:bound} (cf.\ the proof of Theorem~\ref{thm0}). 
For each $i=1,2,\ldots, N$, consider the function
$x \mapsto \frac{G(|x|)}{|x|}x_i$, $x\in \Omega$, and notice that this function is an element of $\tilde{H}^1(\Omega) \setminus \{0\}$. 
Moreover, since $\Omega$ is assumed to be either symmetric of order $2$ or centrally symmetric, \cite[Propositions~B.1 and B.3]{ABD1} (with the measure $\phi_1(x)\,dx$ instead of $dx$, see the discussion above) yield
\begin{equation}\label{eq:phiG1}
\int_{\Omega} \phi_1(x)\frac{G(|x|)}{|x|} x_i \, dx = 0, \quad i=1,\dots,N.
\end{equation}
Thus, each $\frac{G(|x|)}{|x|}x_i$ is a valid trial function for  \eqref{eq:tau2def}, i.e.,
\begin{equation}\label{eq:tau2i}
\tau_2(\Omega) \int_\Omega \frac{G^2(|x|)}{|x|^2}x_i^2\,dx
\leq
\int_\Omega \left|\nabla \left(
\frac{G(|x|)}{|x|} x_i
\right)\right|^2 \,dx
+
h \int_{\partial B_\alpha}
\frac{G^2(|x|)}{|x|^2}x_i^2\,dS,
\quad
i=1,2,\ldots,N.
\end{equation}
Applying \cite[Remark~B.11 with $g(r)=G(r)/r$ and $p=x_i$]{ABD1} to expand the gradient term, and then
summing the inequalities \eqref{eq:tau2i} over $i$, we get
$$
\tau_2(\Omega) \int_\Omega G^2(|x|) \,dx
\leq
\int_\Omega 
\left(
(G'(|x|))^2 + \frac{(N-1)}{|x|^2} G^2(|x|)
\right)
\,dx
+
h \int_{\partial B_\alpha}
G^2(|x|)\,dS.
$$
Dividing by $\int_\Omega G^2(|x|) \,dx$ and applying Proposition~\ref{prop:bound} with $l=1$, we arrive at the desired inequality $\tau_2(\Omega) \leq \tau_{1,1}$, and hence  \eqref{eq:SWND} holds.

\ref{thm1:2}
First, we discuss the relations \eqref{eq:SWNDx}.
As in the assertion~\ref{thm1:1}, thanks to Corollary~\ref{cor:multiplicity}~\ref{cor:multiplicity:1},  it is sufficient to prove that $\tau_{N+1}(\Omega) \leq \tau_{1,1}$ provided  $h \in (-\infty,\bar{h}_2] \cup [h_2,+\infty]$, when $\Omega$ is symmetric of order $4$.
For this purpose, we consider the following subspace of $\tilde{H}^1(\Omega)$:
$$
{Y}_{N}
=
\text{span}
\left\{
\frac{G(|x|)}{|x|}x_1,
\dots,
\frac{G(|x|)}{|x|}x_{N+1}
\right\}.
$$
It is not hard to see that basis elements of ${Y}_{N}$ are mutually linearly independent. 
Moreover, since the symmetry of order $4$ implies the symmetry of order $2$, basis elements of ${Y}_{N}$ are $L^2(\Omega)$-orthogonal to $\phi_1$, see \eqref{eq:phiG1}.
Therefore, $\text{dim}\,{Y}_{N} = N$ and ${Y}_{N}$ is an admissible subspace for the definition \eqref{eq:taukdef} of $\tau_{N+1}(\Omega)$.

Thanks to the symmetry of order $4$, it can be shown that basis elements of ${Y}_{N}$ are mutually $L^2(\Omega)$- and $H_0^1(\Omega)$-orthogonal (i.e., their gradients are mutually $L^2(\Omega)$-orthogonal), cf.\ \cite[Eq.~(3.6)]{ABD1}.
Moreover, in the same way as in \cite{ABD1}, it can be deduced that
$$
\int_{\partial B_\alpha} 
\left(\frac{G(|x|)}{|x|}x_i\right)
\left(\frac{G(|x|)}{|x|}x_j\right) dS
=
0
\quad \text{for any}~ i \neq j.
$$
Then, arguing as in \cite[Section~3.2]{ABD1}, for any element $u \in {Y}_N \setminus \{0\}$ we get
$$
\frac{\int_{\Omega} |\nabla u|^2 \, dx + h \int_{\partial B_\alpha} u^2 \,dS}{\int_{\Omega} u^2 \, dx}
=
\frac{\int_\Omega 
\left(
(G'(|x|))^2 + \frac{(N-1)}{|x|^2} G^2(|x|)
\right)
\,dx
+
h \int_{\partial B_\alpha}
G^2(|x|)\,dS}{\int_\Omega G^2(|x|) \,dx}.
$$
The characterization \eqref{eq:taukdef} of $\tau_{N+1}(\Omega)$ and Proposition~\ref{prop:bound} with $l=1$ yield $\tau_{N+1}(\Omega) \leq \tau_{1,1}$, which gives \eqref{eq:SWNDx}.

\medskip
Second, we discuss the inequality \eqref{eq:SWNDy}.
Applying Corollary~\ref{cor:multiplicity}~\ref{cor:multiplicity:1}, we see that it is sufficient to justify the inequality $\tau_{N+2}(\Omega) \leq \tau_{2,1}$ provided $h \in (-\infty,\bar{h}_2] \cup [h_2,+\infty]$, when $\Omega$ is symmetric of order $4$.
Let $v$ be a positive eigenfunction of the SL problem \eqref{eq:ode}-\eqref{eq:bcM} and $G$ be a function defined correspondingly as in Proposition~\ref{prop:bound} with $l=2$. 
Consider the subspace
$$
{Y}_{N+1}
=
\text{span}
\left\{
\frac{G(|x|)}{|x|}x_1,
\dots,
\frac{G(|x|)}{|x|}x_{N+1}, w
\right\},
$$
where $w$ is defined via a special linear combination of eigenfunctions corresponding to $\tau_{N+2}(B_\beta \setminus \overline{B_\alpha})$ as
\begin{equation}\label{eq:ww}
w 
= 
\sqrt{2} \sum_{i=1}^{N-1} \sum_{j=i+1}^N  \frac{G(r)}{r^2} x_i x_j
+
\sum_{i=1}^{N-1} \frac{G(r)}{\sqrt{i(i+1)}r^2} \left(\sum_{j=1}^i x_j^2 - i x_{i+1}^2\right).
\end{equation}
Arguing as in \cite[Section~3.3]{ABD1} (see also Remark~\ref{rem:thm} below), it can be shown that 
${Y}_{N+1}$ is an admissible subspace for the definition \eqref{eq:taukdef} of $\tau_{N+2}(\Omega)$ and 
\begin{align}
\notag
\tau_{N+2}(\Omega)
&\leq 
\max_{u \in {Y}_{N+1} \setminus \{0\}}\frac{\int_{\Omega} |\nabla u|^2 \, dx + h \int_{\partial B_\alpha} u^2 \,dS}{\int_{\Omega} u^2 \, dx}
\\
\label{eq:taun+2}
&\leq
\frac{\int_\Omega 
\left(
(G'(|x|))^2 + \frac{2N}{|x|^2} G^2(|x|)
\right) dx
+
h \int_{\partial B_\alpha}
G^2(|x|)\,dS}{\int_\Omega G^2(|x|) \,dx}.
\end{align}
Applying Proposition~\ref{prop:bound} with $l=2$ and Corollary~\ref{cor:multiplicity}~\ref{cor:multiplicity:1}, we conclude that $\tau_{N+2}(\Omega) \leq \tau_{2,1} = \tau_{N+2}(B_\beta \setminus \overline{B_\alpha})$ which is the desired inequality  \eqref{eq:SWNDy}.

The equality cases in 
\eqref{eq:SWND}, 
\eqref{eq:SWNDx}, \eqref{eq:SWNDy}
also follow from Proposition~\ref{prop:bound}.
\qed

\begin{remark}\label{rem:thm}
The proof of the counterpart of \eqref{eq:SWNDy} in the planar case $N=2$ with $h=0$ given in \cite[Section~3.3]{ABD1} 
 contains an imprecision. 
Namely, it is based on  application of \cite[Lemma~B.6]{ABD1}, but
the rotation $T \in SO(2)$, whose existence is stated in this lemma, \textit{depends} on the function $g$ and hence might not be common for both integrals in \cite[Eq.~(B.17)]{ABD1}.
In the proof of Theorem~\ref{thm1-ext} given above, 
we overcome this issue by considering a different (and, in fact, simpler) function $w$ (cf. \eqref{eq:wfunc} and \eqref{eq:ww}). 
\end{remark}

\begin{remark}\label{rem:symmetry}
In order to perform the proof of Theorem~\ref{thm1}~\ref{thm1:1} (or Theorem~\ref{thm1-ext} with $\kappa=1$ and $i=2$), it is sufficient to \textit{assume} that
\begin{equation}\label{eq:assump-sym1}
\int_{\Omega} \phi_1(x)\frac{G(|x|) x_i}{|x|} \, dx = 0, \quad i=1,\dots,N,
\end{equation}
instead of imposing the symmetry assumptions on $\Omega$. 
Similar equalities can replace higher symmetry assumptions from
Theorems~\ref{thm1-ext} and~\ref{thm1}, but they are less constructive, and because of this, we refrain from formulating  Theorems~\ref{thm1-ext} and~\ref{thm1} in their terms.
\end{remark}

\section{Counterexamples}\label{sec:counterexample}

In this section, we discuss several examples of planar domains which do not have the required symmetries to apply Theorem~\ref{thm1-ext} and for which the corresponding inequalities~\eqref{eq:SWhigh} are reversed.  

\subsection{Counterexample to \eqref{eq:SWhigh} with $i=2$ when $\kappa=0$}
For simplicity, we restrict ourselves only to the limiting case $h=+\infty$ corresponding to \eqref{eq:D} with the Dirichlet-Neumann boundary conditions.
We provide a construction based on the consideration of a dumbbell-shaped domain defined as follows, see Figure \ref{fig:0}.
%Let $B_1$ be a unit disk centered at the origin.
Denote by $\Omega_L = \Omega_L(\varepsilon)$ a domain which is a $C^1$-smooth perturbation, governed by a sufficiently small $\varepsilon>0$, of the unit disk $B_1(-1-l+\varepsilon,0)$ centered at the point $(-1-l+\varepsilon,0)$, where $l>0$ is fixed.
We assume that $\Omega_L$ has the following properties:
\begin{enumerate}
	\item $\Omega_L$ is symmetric with respect to the line $y=0$;
	\item $\Omega_L \subset \{(x,y): x<-l\}$ and $|\Omega_L| = |B_1|$;
	\item $\partial \Omega_L \cap \{(x,y): x = -l\}$ is an interval of length $\varepsilon$.
\end{enumerate}
Denote by $\Omega_R = \Omega_R(\varepsilon)$  the reflection of $\Omega_L$ with respect to the line $x=0$.
Let $0<\de<\varepsilon/2$ and let $T_\delta$ be the rectangle $[-l,l] \times (-\delta,\delta)$ connecting $\Omega_L$ with $\Omega_R$.
Finally, we set $\Omega_\delta = \Omega_L \cup T_\delta \cup \Omega_R$.
It is clear from the construction that $\Omega_\delta$ is a centrally symmetric domain.
\begin{figure}[ht!]
	\centering
	\includegraphics[width=0.75\linewidth]{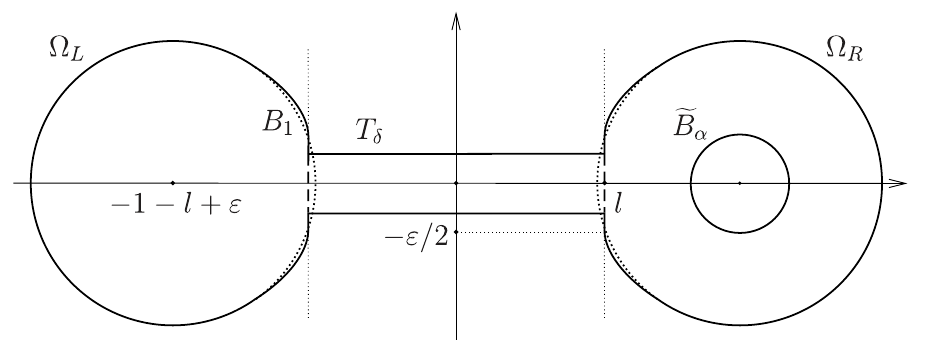}\\
	\caption{The reference domain $\Omega_\delta$ and the disk $\widetilde{B}_{\alpha} = \overline{B_\alpha(1+l-\varepsilon,0)}$ inside it.}
	\label{fig:0}
\end{figure}

Hereinafter, the spectrum of the one-dimensional Dirichlet Laplacian on a segment of length $2l$ will be denoted by $\{\lambda_m(l)\}_{m \in \mathbb{N}}$, and $\{\mu_r(\Omega_L)\}$ will stand for the spectrum of the Neumann Laplacian in $\Omega_L$.

Taking any $\alpha\in (0,1-\varepsilon)$, we set $\widetilde{B}_{\alpha} = \overline{B_\alpha(1+l-\varepsilon,0)}$ and 
consider the mixed eigenvalue problem \eqref{eq:D} in $\Omega_\delta \setminus \widetilde{B}_{\alpha}$ with the Dirichlet boundary condition on $\partial\widetilde{B}_{\alpha}$
and the Neumann boundary condition on $\partial \Omega_\delta$. 
Evidently, $\Omega_\delta \setminus \widetilde{B}_{\alpha}$ is not symmetric of order $2$.
Let  $B_{\beta(\delta)}$ be a disk (centered at the origin) of the same area as $\Omega_\delta$.
In order to obtain a counterexample to the inequality \eqref{eq:SWhigh}, we compare the values of
\begin{equation}\label{eq:counterxm18a}
\tau_2(\Omega_{\delta} \setminus \widetilde{B}_{\alpha}) 
\quad \text{and} \quad 
\tau_2(B_{\beta(\delta)} \setminus \overline{B_{\alpha}})
\end{equation}
for a sufficiently small $\delta$.  
Notice that, by construction, 
\begin{equation}\label{eq:conv2}
|B_{\beta(\delta)} \setminus \overline{B_{\alpha}}| = |\Omega_\delta \setminus \widetilde{B}_{\alpha}| 
= 
|\Omega_L| + |T_\delta|+  |\Omega_R \setminus \widetilde{B}_{\alpha}|  
\to
|B_{1}| + |B_1 \setminus \overline{B_{\alpha}}|
\quad \text{as}~ \delta \to 0,
\end{equation}
which implies that $\beta(\delta) \to \sqrt{2}$ as $\delta \to 0$.
Moreover, it is known (see, e.g.,  \cite[Theorem~2.2]{arrieta} for an explicit reference) that
\begin{equation}\label{eq:counter_conv2}
\tau_k(\Omega_\delta \setminus \widetilde{B}_{\alpha}) \to \widetilde{\tau}_k(\varepsilon)
\quad \text{as}~
\delta \to 0
\end{equation}
for any $k \in \mathbb{N}$, 
where 
\begin{equation}\label{eq:counter_union_2}
\{\widetilde{\tau}_k(\varepsilon)\}_{k \in \mathbb{N}} 
=
\{\mu_r(\Omega_L)\}_{r \in \mathbb{N}} 
\cup 
\{\lambda_m(l)\}_{m \in \mathbb{N}}
\cup
\{\tau_d(\Omega_R \setminus \widetilde{B}_{\alpha})\}_{d \in \mathbb{N}},
\end{equation}
and the set on the right-hand side of \eqref{eq:counter_union_2} is arranged in increasing order (counting multiplicities). 
Noting that $\widetilde\tau_{1}(\varepsilon)=\mu_1(\Omega_L)=0,$ we take $l>0$ small enough so that $\lambda_1(l) > \max\left\{\mu_2(\Omega_L),\tau_1(\Omega_R \setminus \widetilde{B}_\alpha)\right\}$. Therefore, we get
\begin{equation}\label{eq:counter_conv1x}
\widetilde{\tau}_2(\varepsilon) 
=
\min\left\{
\mu_2(\Omega_L),
\tau_1(\Omega_R \setminus \widetilde{B}_\alpha)
\right\}.
\end{equation}
Moreover, recalling that $\Omega_L$ and $\Omega_R$ are sufficiently small $C^1$-smooth perturbations of the unit disk $B_1(-1-l+\varepsilon,0)$, we have
\begin{equation}\label{eq:smoothpert1}
\mu_2(\Omega_L) \to \mu_2(B_{1}) 
\quad 
\text{and}
\quad
\tau_1(\Omega_R \setminus \widetilde{B}_\alpha) 
\to 
\tau_1(B_1 \setminus \overline{B_{\alpha}}) \quad \text {as } \varepsilon \to 0,
\end{equation}
see, e.g., \cite[Section VI.2.6]{courant}.
Thus, by the continuity, the comparison of values of  $\tau_2(\Omega_{\delta} \setminus \widetilde{B}_{\alpha})$ and $\tau_2(B_{\beta(\delta)} \setminus \overline{B_{\alpha}})$ for sufficiently small $\varepsilon$ and $\delta$ is reduced to the comparison of values of
$$
\mu_2(B_1),
\quad
\tau_1(B_1 \setminus \overline{B_{\alpha}}),
\quad
\tau_2(B_{\sqrt{2}} \setminus \overline{B_{\alpha}}).
$$
Each of these values can be computed with arbitrary precision by means of Bessel functions.
In particular, it is known that $\mu_2(B_1) \approx 3.38997$. 
On Figure \ref{fig:4}, we depict the dependence of $\tau_1(B_1 \setminus \overline{B_{\alpha}})$ and $\tau_2(B_{\sqrt{2}} \setminus \overline{B_{\alpha}})$ on $\alpha \in (0,1)$, which shows that 
$$
\min\left\{
\mu_2(B_1),
\tau_1(B_1 \setminus \overline{B_{\alpha}})
\right\}
>
\tau_2(B_{\sqrt{2}} \setminus \overline{B_{\alpha}})
$$
for any $\alpha \in [0.2,0.6]$. 
This means that
$$
\tau_2(\Omega_{\delta} \setminus \widetilde{B}_{\alpha}) 
>
\tau_2(B_{\beta(\delta)} \setminus \overline{B_{\alpha}})
$$
for such $\alpha$ and all sufficiently small $\varepsilon$ and $\delta$, which is a contradiction to \eqref{eq:SWhigh} for $i=2$ when $\kappa=0$.

\begin{figure}[ht!]
	\centering
	\includegraphics[width=0.6\linewidth]{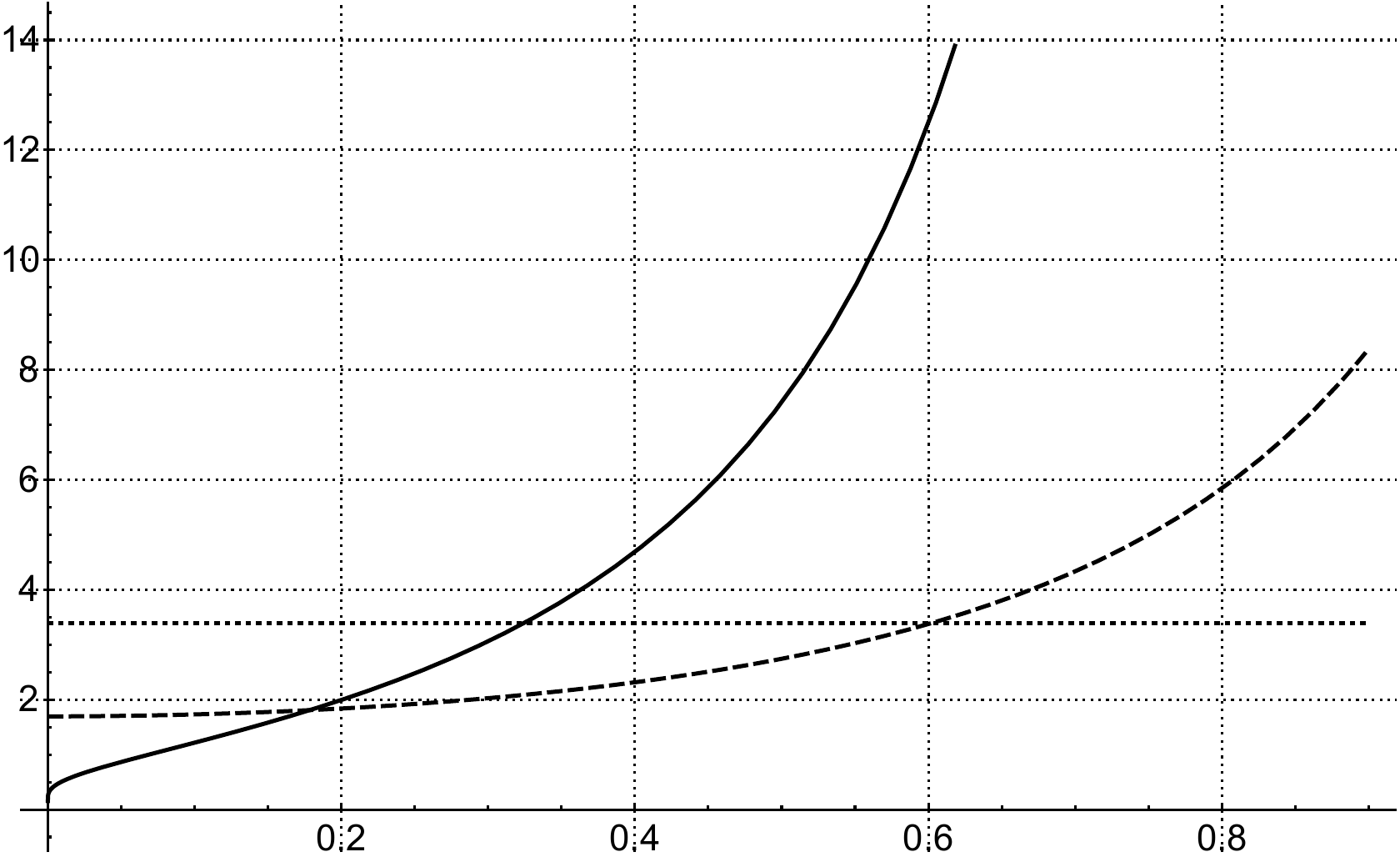}\\
	\caption{The dependence of $\tau_1(B_1 \setminus \overline{B_{\alpha}})$ (solid line) and $\tau_2(B_{\sqrt{2}} \setminus \overline{B_{\alpha}})$ (dashed line) on $\alpha \in (0,1)$;  $\mu_2(B_1) \approx 3.38997$ (dotted line).}
	\label{fig:4}
\end{figure}

\subsection{Counterexamples to \eqref{eq:SWhigh} with $i=3,4$ when $\kappa=1$, and $i=5$ when $\kappa=2$}\label{counter:3}
Assume that $h \in (0,+\infty)$.
We argue in much the same way as in \cite[Sections~4.2 and 4.3]{ABD1}.
Consider a rectangle $\Omega_{a} = \left(-\frac{a}{2},\frac{a}{2}\right)\times \left(-\frac{1}{2a},\frac{1}{2a}\right)$ with $a>0$.
 Clearly, we have $|\Omega_{a}|=1$, and if $a \neq 1$, then $\Omega_a$ is symmetric of order $2$ but not symmetric of order $4$.
 By the classical theory, the spectrum $\{\mu_r(\Omega_{a})\}$ of the Neumann Laplacian in $\Omega_{a}$ is exhausted by eigenvalues $\pi^2 k^2/a^2+\pi^2 m^2 a^2$, where $k,m \in \mathbb{N}_0$.
 It can be computed that 
 \begin{align}
 \label{eq:mu3}
 \mu_3(\Omega_{\sqrt{3}}) = \frac{4\pi^2}{3} \approx 13.1594 > \mu_2(B_{1/\sqrt{\pi}}) \approx 10.6499,\\
 \label{eq:mu4}
 \mu_4(\Omega_{\sqrt{3}}) = 3\pi^2 \approx 29.6088 > \mu_4(B_{1/\sqrt{\pi}}) \approx 29.3059,
 \end{align}
 where $B_{1/\sqrt{\pi}}$ is a disk of radius $1/\sqrt{\pi}$, so that $|B_{1/\sqrt{\pi}}|=1$. 
 On the other hand, it is known that 
 \begin{equation}\label{eq:counter2}
 \tau_k(\Omega \setminus \overline{B_{\alpha}}) \to \mu_k(\Omega)
 \quad \text{as}~ \alpha \to 0
 \end{equation}
 for any $k \in \mathbb{N}$  
 and $h \in (0,+\infty)$, see e.g., \cite[Corollary~2.2]{BS}, and also \cite{BM} for a related discussion in the higher-dimensional case.
 Therefore, combining \eqref{eq:mu3} (reps.~\eqref{eq:mu4}) and \eqref{eq:counter2} (with $\Omega = \Omega_{\sqrt{3}}$ and $\Omega=B_\beta$), we provide a counterexample to \eqref{eq:SWhigh} for $i=2$  (resp.~$i=3$) when $\kappa=1$ for all sufficiently small $\alpha>0$.
 
 Finally, we discuss the inequality  \eqref{eq:SWhigh} for $i=5$ when $\kappa=2$. 
 It is known from \cite[Section~5.4]{hersh2} that 
$$
\mu_5(\Omega_1) = 4 \pi^2 \approx 39.4784
> \mu_4(B_{1/\sqrt{\pi}}) = \mu_5(B_{1/\sqrt{\pi}}) \approx 29.3059,
$$
where we recall that $\Omega_1$ is a square.
Using the convergence \eqref{eq:counter2}, we deduce as above that the inequality \eqref{eq:SWhigh} does not hold, in general, if $\Omega_{\textnormal{out}} \setminus \overline{B_\al}$ has only the symmetry of order $4$.

\section{Final comments and remarks}\label{sec:final}
Let us collect in one place a few remarks that naturally arise in the present work and which might be interesting for readers.

\begin{enumerate}[label={\rm(\arabic*)}]
	\item It deserves further investigation whether the results of Theorems~\ref{thm0}, \ref{thm1-ext}, \ref{thm1} remain valid for domains with ``holes'' of a more general shape, especially in the case $h<0$.
	We refer to \cite{ABD1,AK,DPP,hersh0} for related results in this direction, which are mainly devoted to the case $h \in [0,+\infty]$. 
\item\label{rem:2} It would be interesting to know whether some of the inequalities 
\eqref{eq:SWhigh} (with $\kappa=1,2$),
\eqref{eq:SWND}, \eqref{eq:SWNDx}, \eqref{eq:SWNDy} remain valid for any $h \in \mathbb{R} \cup \{+\infty\}$ and $0<\alpha<\beta$. 
	Note that the restriction on the range of $h$ in Theorems~\ref{thm1-ext},~\ref{thm1} is dictated by the method of the proof which is not directly applicable if the second eigenfunction is radially symmetric (see Remark~\ref{rem:radial} on the possibility of radial  symmetry of the second eigenfunction), and hence the restriction might be merely technical.
\item In continuation of the previous remark, it would be interesting to investigate whether an example as in Remark~\ref{rem:radial} takes place in all dimensions $N \geq 3$. 
	For instance, if, for some $N \geq 3$, $\tau_{1,1} < \tau_{0,2}$ for all $h \in \mathbb{R} \cup \{+\infty\}$ and $0<\alpha<\beta$ (i.e., the second eigenfunction of \eqref{eq:D} in $B_\beta \setminus \overline{B_\alpha}$ is always nonradial and, more precisely, antisymmetric with respect to a central section of $B_\beta \setminus \overline{B_\alpha}$), then the inequality \eqref{eq:SWND} is valid for the same ranges of $h,\alpha$. 
\item Figure~\ref{fig:thm1} represents only a schematic graph, and we do not investigate the behavior of $h_2$ and $\bar{h}_2$ in the present work. 
In particular, it is interesting to know whether the union of graphs of $h_2$ and $\bar{h}_2$ forms a closed curve that is separated from the $h$-axis (in contrast to the behavior depicted in Figure~\ref{fig:thm1}). 
    \item The symmetry of order $2^\kappa$ required in  Theorem~\ref{thm1-ext} is only a sufficient assumption.
	It would be interesting to obtain an inequality as \eqref{eq:SWhigh} for some higher indices under different geometric assumptions on $\Omega$, cf.\ Remark~\ref{rem:symmetry}.
\end{enumerate}

\appendix
\section{}\label{sec:appendix}
In this section, we prove Lemmas~\ref{lem:v} and \ref{lem:n+2} stated in Section~\ref{sec:spectrum}, as well as several auxiliary results. 
Occasionally, we will use the extended notation $\tau_{l,j}(h)$, $\tau_{l,j}(\alpha)$, and $\tau_{l,j}(h,\alpha)$, in order to represent the dependence of an eigenvalue $\tau_{l,j}$ of the SL problem \eqref{eq:ode}-\eqref{eq:bcM} on the parameters  $h \in \mathbb{R} \cup \{+\infty\}$ and $\alpha \in (0,\beta)$ (assuming that $\beta>0$ is fixed). 
We always assume that $N \geq 2$, and that $N,l$ are natural numbers.
However, most of the results from this section remain valid under more general assumptions on $N$ and $l$, since $l(l+N-2)$ is a coefficient in the equation~\eqref{eq:ode}.

We start with the following simple lemma.
\begin{lemma}\label{lem:vneq0}
	Let $h \in \mathbb{R} \cup \{+\infty\}$, $0<\alpha<\beta$, and $l \in \mathbb{N}_0$.
	Let $v$ be an eigenfunction of the SL problem \eqref{eq:ode}-\eqref{eq:bcM}. 
	Then the following assertions are satisfied:
	\begin{enumerate}[label={\rm(\roman*)}]
		\item\label{lem:vneq0:1} 
		If $h<+\infty$, then $v(\alpha) \neq 0$.
		\item \label{lem:vneq0:3}  $v(\beta) \neq 0$. 
  %\tb{to delete?}
  \item\label{lem:vneq0:2} 
		If $h \neq 0$, then $v'(\alpha) \neq 0$.
\end{enumerate}
   In particular, if $v>0$ in $(\al,\ga)$ for some $\ga \in (\al,\beta)$, then 
   \begin{enumerate}[label={\rm(\alph*)}]
       \item\label{lem:vneq0:a} $v(\al)>0$ for $h<+\infty,$
       \item\label{lem:vneq0:b} $v'(\al)>0$ for $0<h\le +\infty,$
       \item\label{lem:vneq0:c} $v'(\al)<0$ for $h<0.$
   \end{enumerate}
\end{lemma}
\begin{proof}
    \ref{lem:vneq0:1} Let $h<+\infty$. 
    Suppose, by contradiction, that $v(\alpha)=0$. Then we have $v'(\alpha) = h v(\alpha) = 0$ by the Robin boundary condition. 
    On the other hand, by the standard Sturm--Liouville theory, $v$ changes sign a finite number of times, and hence we may assume that $v\le 0$ 
    in an interval $(\alpha,\gamma)$ for some $\gamma \in (\alpha,\beta)$. 
    Thus, $\alpha$ is a point of maximum of $v$ over $(\alpha,\gamma)$. 
    Moreover, in view of \eqref{eq:ode}, $v$ satisfies 
    $$
    v'' + \frac{N-1}{r} v' -\frac{l(l+N-2)}{r^2} v = -\tau v
     \ge 0 ~\text{ in }~ (\alpha,\gamma), ~\text{ if }~ \tau\ge 0,
    $$
    and
    $$
     v'' + \frac{N-1}{r} v' 
     +
     \left(-\frac{l(l+N-2)}{r^2} +\tau\right) v
    = 0 ~\text{ in }~ (\alpha,\gamma), ~\text{ if }~ \tau <  0.
    $$
    However, these facts contradict the boundary point lemma, see, e.g., \cite[Chapter~1, Theorem~4]{ProtterWeinberger}.

    \ref{lem:vneq0:3}
    Noting that $v'(\beta)=0$ by the Neumann boundary condition, 
    the claim can be proved using the boundary point lemma as in the proof of the assertion~\ref{lem:vneq0:1}.
        
    \ref{lem:vneq0:2} 
    First, we assume that $h=+\infty$. This corresponds to the Dirichlet boundary condition $v(\alpha)=0$, and we apply again \cite[Chapter~1, Theorem~4]{ProtterWeinberger} to deduce that $v'(\alpha) \neq 0$. Now 
     assume that $h < +\infty$.
	By the assertion~\ref{lem:vneq0:1}, we have
     $v(\alpha) \neq 0$, and hence $v'(\alpha) = h v(\alpha) \neq 0$ whenever $h \neq 0$, thanks to the Robin boundary condition. 
     
     Finally, assume that $v>0$ in a right neighborhood of  $\al$. 
     Then \ref{lem:vneq0:a} directly follows from \ref{lem:vneq0:1}.  
     For $h<+\infty,$ \ref{lem:vneq0:b} and \ref{lem:vneq0:c}  follow from the Robin boundary condition and \ref{lem:vneq0:a}.
     For $h=+\infty,$ \ref{lem:vneq0:b}  follows  from \ref{lem:vneq0:2}.
\end{proof}

The following auxiliary lemma is a consequence of, e.g., \cite[Theorem~4.2:1]{KZ} (see also \cite[Lemma~2.11]{AFKen} for a related result), and we provide a proof for clarity.
\begin{lemma}\label{lem:tau-derivative}
	Let $h \in \mathbb{R}$, $0<\alpha<\beta$, $l \in \mathbb{N}_0$, and $j \in \mathbb{N}$.
	Let $v$ be an eigenfunction of the SL problem \eqref{eq:ode}-\eqref{eq:bcM} corresponding to the eigenvalue $\tau_{l,j}(h)$ and normalized as
	\begin{equation}\label{eq:normalization1}
		\int_\alpha^\beta r^{N-1} v^2 \,dr = 1.
	\end{equation}
	Then the mapping $h \mapsto \tau_{l,j}(h)$ is continuously differentiable
	and
	\begin{equation}\label{eq:tau-derivative}
	\frac{d \tau_{l,j}(h)}{d h}
	=
	\alpha^{N-1} v^2(\alpha) > 0.
	\end{equation}
 Moreover,
 \begin{equation}\label{eq:tauincr}
 \tau_{l,j}(h) \nearrow \tau_{l,j}(+\infty) \quad \text{as}~ h \to +\infty,
 \end{equation}
 where $\tau_{l,j}(+\infty)$ is the eigenvalue of the SL problem \eqref{eq:ode}-\eqref{eq:bcM} with the Dirichlet boundary condition at $\alpha$.
\end{lemma}
\begin{proof}
    Let us parameterize the Robin boundary  condition at $\alpha$ as
    $-v'(\alpha)+\frac{\cot(\theta)}{p(\al)}\,v(\alpha)=0$, where $\theta\in (0,\pi)$
    and $p(\al)=\alpha^{N-1}$.
    That is, we set  $h=h(\theta)=\frac{\cot(\theta)}{p(\al)}$, and hence $\theta(h)=\cot^{-1}(p(\alpha)h)$ and
    $$
    \frac{d \theta }{d h}=-\frac{p(\alpha)}{p^2(\alpha)h^2+1}.
    $$
    Since the eigenvalue $\tau_{l,j}$ is simple by the standard Sturm--Liouville theory (see, e.g., \cite[Section~13]{weidmann}), we apply \cite[Theorem~4.2:1]{KZ} to get
    $$
    \frac{d \tau_{l,j}}{d \theta}
    =
    -v^2(\al)-\left(p(\al)v'(\alpha)\right)^2.
    $$ 
    Using now the Robin boundary condition, we obtain
    \begin{align*}
    \frac{d \tau_{l,j}}{d h}
	&=
    \frac{d \tau_{l,j}}{d \theta}\cdot \frac{d \theta }{d h}
    =
    \left(v^2(\al)+\left(p(\al)v'(\alpha)\right)^2\right)\frac{p(\alpha)}{p^2(\alpha)h^2+1}\\
 &=v^2(\al)(1+p^2(\al)h^2)\frac{p(\alpha)}{p^2(\alpha)h^2+1}
 =
 p(\alpha)v^2(\al)
 =
 \al^{N-1}v^2(\al).
 \end{align*}
 This completes the proof of the equality in \eqref{eq:tau-derivative}.
 The inequality in \eqref{eq:tau-derivative} follows from Lemma~\ref{lem:vneq0}~\ref{lem:vneq0:1}.

  Finally, the convergence \eqref{eq:tauincr} follows from \cite[Theorem~(a):(i)]{EMZ}.
\end{proof}

Let us now prove Lemma~\ref{lem:v}.
\begin{proof}[Proof of Lemma \ref{lem:v}]
Throughout the proof, to emphasize the dependence on $h$, we use the notation $v_h$ to denote the positive eigenfunction $v$ of the SL problem \eqref{eq:ode}-\eqref{eq:bcM}.
In particular, $v_h$ satisfies the following equation (see \eqref{eq:ode3}):
\begin{equation}\label{eq:ode3proof1}
(r^{N-1}v_h'(r))'
=
\left(
\frac{l(l+N-2)}{r^2}-\tau_{l,1}(h)
\right)
v_h(r)r^{N-1}, \quad  r\in (\al,\beta).
\end{equation}

The assertion~\ref{eq:long:1} is simple.
It is easy to see from \eqref{eq:taudef-first} with $l=0$ that $\tau_{0,1}(0) = 0$ by considering a nonzero constant trial function. 
Then, noting that the mapping $h \mapsto \tau_{0,1}(h)$ is  increasing by Lemma~\ref{lem:tau-derivative}, 
we get the desired conclusion about the signs of $\tau_{0,1}(h)$.
Integrating now both sides of \eqref{eq:ode3proof1} with $l=0$ over $(r,\beta)$ for $r \in (\alpha,\beta)$, we get
\begin{equation}\label{eq:ode3l02}
\tau_{0,1}(h)\int_r^\beta s^{N-1} v_h \,ds
=
-
\int_r^\beta (s^{N-1}v_h')' \,ds
=
r^{N-1}v_h'(r) 
-
\beta^{N-1}v_h'(\beta)
=
r^{N-1}v_h'(r).
\end{equation}
Thus, we conclude that
$v_h'<0$ for $h<0$, $v_h'=0$ for $h=0$, and $v_h'>0$ for $h>0$.

Let us now discuss the assertion~\ref{eq:long:2}:\ref{eq:long:2:iii}.  
Since  $l \geq 1$, the assertion~\ref{eq:long:1}:\ref{eq:long:1:ii} and the chain of inequalities \eqref{eqn:col} yield $\tau_{l,1}(0)>0$. 
In view of Lemma~\ref{lem:tau-derivative}, the mapping $h \mapsto \tau_{l,1}(h)$ is continuous and increasing.
Therefore, 
we can find $\de>0$ such that 
$\tau_{l,1}(h)>0$ whenever $h>-\de$.
On the other hand, taking a constant trial function in the definition \eqref{eq:taudef-first} of $\tau_{l,1}$, we see that $\tau_{l,1}(h)<0$ when $h$ takes a sufficiently large negative value.
Now setting 
$$
h_1 = \inf\{h < 0:~ \tau_{l,1}(h) > 0\},
$$
we see that $h_1>-\infty$ and it has all the required properties.

To prove the remaining assertions, we
denote, for brevity, $f_h(s) =
\frac{l(l+N-2)}{s^2}-\tau_{l,1}(h)$ for $s\in [\al,\be)$ and remark that $f_h$ is decreasing in $[\al,\be)$.
Integrating  \eqref{eq:ode3proof1} over $(r,\beta)$, we get
\begin{equation}\label{eq:-r=int0}
-r^{N-1} v_h'(r) 
= 
\int_r^\beta 
f_h(s)
v_h(s)s^{N-1} \,ds, \quad r\in [\al,\be).
\end{equation}
Now we make the following three observations: 

\noindent 
\textit{Observation 1}: 
$v_h'<0$ in $(\al,\be)$  for  $h<h_1.$ 
This follows from \eqref{eq:-r=int0} by noting that $\tau_{l,1}(h) <0$ and hence $f_h>0$ for $h<h_1$  (see the assertion \ref{eq:long:2}:\ref{eq:long:2:iii}).
 
 \noindent 
 \textit{Observation 2}: If $v_h'(s_0) \geq 0$  for some $s_0\in [\al,\be)$, then $v_h'(r)>0$ for $r\in (s_0,\be).$ 
 To show this, first, we use \eqref{eq:-r=int0} and conclude that $f_h$ must be nonpositive somewhere in $(s_0,\be).$ 
 Since $f_h$ is decreasing, we can find the smallest real number $s_1\in[\al,\be)$ such that $f_h(s)<0$ for $s> s_1.$  
 Therefore, using  \eqref{eq:-r=int0}, we get $v_h'(r)>0$ for  $r\in [s_1,\be)$.  
 If $s_1 \leq s_0$, then we are done.
 If $s_1>s_0,$ then $f_h(s_1)=0$ and $f_h(s)>0$ for $s\in [\al,s_1)$.
 Thus, for $r\in (s_0,s_1)$ we have
\begin{align}
    -r^{N-1} v_h'(r)=&
\int_r^{s_1} 
f(s)
v_h(s)s^{N-1} \,ds+\int_{s_1}^\beta 
f(s)
v_h(s)s^{N-1} \,ds\\
<&
\int_{s_0}^{s_1} 
f(s)
v_h(s)s^{N-1} \,ds+\int_{s_1}^\beta 
f(s)
v_h(s)s^{N-1} \,ds=-{s_0}^{N-1} v_h'(s_0) \leq 0.
\end{align}
Consequently, $v_h'(r)>0$ for $r\in (s_0,s_1)$ and hence for $r\in (s_0,\beta)$.
 
\noindent 
\textit{Observation 3}: If $v_{h^*}'(r)>0$  for some  $h^*\in (-\infty,+\infty]$ and some $r\in [\al,\be)$, then there exist  $\de>0$ and $s^*\in (\al,\be)$ such that $v_h'(s^*)>0$ for $h\in (h^*-\de,+\infty].$ 
For this, we argue as follows.
Since $v_{h^*}'(r)>0,$ we deduce from \eqref{eq:-r=int0} the existence of $s^*\in(\al,\be)$ such that 
$f_{h^*}(s^*) < 0$.
In view of the continuity and monotonicity of the mapping $h \mapsto \tau_{l,1}(h)$ (see Lemma~\ref{lem:tau-derivative}), we can find $\de>0$ such that 
$f_h(s^*) < 0$
whenever $h\in (h^*-\de,+\infty]$. 
Thus, by the monotonicity of $f_h$, \eqref{eq:-r=int0} yields $v_h'(s^*) > 0$ for $h\in(h^*-\de,+\infty]$.

\smallskip
The assertion~\ref{eq:long:2}:\ref{eq:long:2:i} 
follows from Observation~2 by noting that $v_h'(\alpha) \geq 0$, see Lemma~\ref{lem:vneq0}~\ref{lem:vneq0:b} in the case $h>0$.

Next, we consider the assertion~\ref{eq:long:2}:\ref{eq:long:2:ii}. 
We have $v_h'>0$ in $(\alpha,\beta)$ for $h=0$ 
by the assertion~\ref{eq:long:2}:\ref{eq:long:2:i}.
Therefore, thanks to Observation~3, there exist $\de>0$ and $s^*\in (\al,\be)$ such that $v_h'(s^*)>0$ for $h \in (-\de,+\infty].$ 
On the other hand, we have  $v_h'(\al)<0$ for $h < 0$, see Lemma~\ref{lem:vneq0}~\ref{lem:vneq0:c}.
We conclude that $v_h'$ changes sign in $(\alpha,\beta)$ for $h \in (-\de,0)$.
Now we define
\begin{equation}\label{eq:h0}
h_0
=
\inf\{h<0:~ v_h' \text{ changes sign in } (\al,\be) \}.
\end{equation}
It is easy to see from Observation~1 that $h_0 \geq h_1$. 
Moreover, we conclude from Observation~3 that $v_h'<0$ in $(\alpha,\beta)$ for $h\in(-\infty, h_0]$  and $v_h'$ changes sign in $(\al,\be)$ for $h\in (h_0,0).$
Let us now take any ${h} \in (h_0,0).$  
Using Observation~2,  we see that $v_h'$ changes sign exactly once from negative to positive.
As a consequence,  there exists a unique $\gamma=\gamma(l,h) \in (\alpha,\beta)$ such that $v_h'(r)<0$ for $r \in (\alpha,\gamma)$, $v_h'(\gamma)=0$, and $v_h'(r)>0$ for $r \in (\gamma,\beta)$.

\smallskip
Finally, 
let us prove the inequality \eqref{eq:long}, i.e.,
$$
f_h(r) v_h^2(r) \geq f_h(\beta) v_h^2(\beta),
\quad r \in (\alpha,\beta).
$$
In the simplest case, $l=0$, this inequality reads as
\begin{equation}\label{eq:longl0}
-\tau_{0,1}(h) v_h^2(r)
	\geq
-\tau_{0,1}(h) v_h^2(\beta),
	\quad r \in (\alpha,\beta),
\end{equation}
and its validity directly follows from the assertion~\ref{eq:long:1}.

For $l \geq 1$, we  consider the following three cases: 
$$
(i)~ h\ge 0, 
\qquad 
(ii)~ h_0<h<0, 
\qquad 
(iii)~ h\le h_0,
$$
where $h_0$ is given by the assertion~\ref{eq:long:2}:\ref{eq:long:2:ii}, see \eqref{eq:h0}.

 $(i)$ We have $v_h'>0$ by the assertion~\ref{eq:long:2}:\ref{eq:long:2:i}.
 Thus, we deduce from \eqref{eq:-r=int0} and the monotonicity of $f_h$ that $f_h(\be)<0$. 
 Now, for those $r \in (\alpha,\beta)$ for which $f_h(r)\le 0,$ \eqref{eq:long} holds in view of the monotonicity of $f_h$ and $v_h$. For the remaining values of $r \in (\alpha,\beta)$,  \eqref{eq:long} holds trivially since the left-hand side of \eqref{eq:long} is positive, while the right-hand side is negative.

$(ii)$ By the assertion~\ref{eq:long:2}:\ref{eq:long:2:ii}:\ref{eq:long:2:ii:a}, 
there exists $\ga \in (\alpha,\beta)$ such that $v_h'<0$ in $(\al,\ga),$ $v_h'(\ga)=0$, and $v_h'>0$ in $(\ga,\be).$ 
We deduce from \eqref{eq:-r=int0} and the monotonicity of $f_h$ that $f_h(\be)<0$ and $f_h>0$ in $(\al,\ga).$ 
Thus, we have $f_h^{-1}(-\infty,0)\subset (\ga,\be)$, and hence \eqref{eq:long} follows by the the same arguments as given in $(i)$.

 $(iii)$ We have  $v_h'<0$ by the assertion~\ref{eq:long:2}:\ref{eq:long:2:ii}:\ref{eq:long:2:ii:b}. 
 Thus, we conclude from \eqref{eq:-r=int0} and the monotonicity of $f_h$ that $f_h>0$ in $(\al,\be).$ 
 Therefore, the mapping $r \mapsto f_h(r) v_h^2(r)$ is decreasing in $(\alpha,\beta)$ and we easily obtain \eqref{eq:long}.
\end{proof}

For proving  Lemma~\ref{lem:n+2}, we need to establish some additional properties of eigenvalues and eigenfunctions of the SL problem \eqref{eq:ode}-\eqref{eq:bcM}.

\begin{lemma}\label{lem:tau2>0}
	Let $h \in \mathbb{R} \cup \{+\infty\}$ and $0<\alpha<\beta$.
	Then $\tau_{0,2}>0$.
\end{lemma}
\begin{proof}
It is known from the standard Sturm--Liouville theory that an eigenfunction $v$ associated with $\tau_{0,2}$ changes sign in $(\alpha,\beta)$ exactly once, that is, there exists a unique $\gamma \in (\alpha,\beta)$ such that $v(\gamma)=0$.
Since $v$ satisfies \eqref{eq:ode3} with $l=0$ and $\tau=\tau_{0,2}$, i.e., 
\begin{equation}\label{eq:ode3l0}
	-(r^{N-1}v')' = \tau_{0,2} r^{N-1} v, 
	\quad r \in (\alpha,\beta),
\end{equation}
we multiply \eqref{eq:ode3l0} by $v$ and integrate over the interval $(\gamma,\beta)$.
Thanks to the boundary conditions $v(\gamma)=0$ and $v'(\beta)=0$, the integration by parts yields
$$
\int_\gamma^\beta r^{N-1}(v')^2 \,dr = \tau_{0,2} \int_\gamma^\beta r^{N-1}v^2 \,dr.
$$
Since the integrals on both sides are positive, we conclude that $\tau_{0,2}>0$.
\end{proof}

\begin{lemma}\label{lem:compareuv}
    Let $h \in [0, +\infty)$, $0<\alpha<\beta$, and $l \in \mathbb{N}$. 
	%\tb{the case $l=0$?}
	Let $u$ and $v$ be eigenfunctions of the SL problem \eqref{eq:ode}-\eqref{eq:bcM} corresponding to the eigenvalues $\tau_{l,1}$ and $\tau_{0,2}$, respectively, and normalized as follows:
	\begin{equation}\label{eq:normalization0}
	\int_\alpha^\beta r^{N-1} u^2 \,dr = 1
		\quad \text{and} \quad 
	 \int_\alpha^\beta r^{N-1} v^2 \,dr = 1.
	\end{equation}
	If $\tau_{l,1} \leq \tau_{0,2}$, then $0< |u(\alpha)| < |v(\alpha)|$.
\end{lemma}
\begin{proof}
	We know from Lemma~\ref{lem:vneq0}~\ref{lem:vneq0:1} that $u(\alpha) \neq 0$ and $v(\alpha) \neq 0$. 
	Consider the following normalized functions:
	\begin{equation}\label{eq:normal11}
	\bar{u} = \frac{u}{u(\al)}
	\quad \text{and} \quad 
	\bar{v} = \frac{v}{v(\al)}.
	\end{equation}
Thus, $\bar{u}(\alpha)= \bar{v}(\alpha) = 1$. 
We see that $\bar{u}$ is positive in $(\al,\be)$ since $\bar{u}$ is the first eigenfunction.
	Moreover, since $\bar{v}$ is the second eigenfunction, it changes sign exactly once. 
Hence, 
there exists $\sigma \in (\alpha,\beta)$ such that $\bar{v}(r)>0$ for $r\in (\alpha,\sigma)$.  
	
	Recall that $\bar{u}$ and $\bar{v}$ satisfy
	\begin{equation}\label{eq:eqs21}
		-(r^{N-1} \bar{u}')' + l(l+N-2)r^{N-3} \bar{u} = \tau_{l,1} \, r^{N-1} \bar{u}
		\quad \text{and} \quad
		-(r^{N-1} \bar{v}')' = \tau_{0,2} \, r^{N-1} \bar{v},
	\end{equation}
	respectively (see \eqref{eq:ode3}). 
	Multiplying the first equation by $\bar{v}$ and the second equation by $\bar{u}$, and then subtracting one equation from another, we obtain
	\begin{align*}
		-(r^{N-1} \bar{u}')' \bar{v} &+ (r^{N-1} \bar{v}')' \bar{u}\\
		&\equiv
		-(r^{N-1}(\bar{u}' \bar{v} - \bar{u}\bar{v}'))' 
		=
		(\tau_{l,1} - \tau_{0,2}) r^{N-1} \bar{u} \bar{v} - l(l+N-2)r^{N-3} \bar{u}\bar{v}.
	\end{align*}
	Our assumption $\tau_{l,1} \leq  \tau_{0,2}$ and the sign properties of $\bar{u}$ and $\bar{v}$ imply that $r^{N-1}(\bar{u}' \bar{v} - \bar{u}\bar{v}')$ is increasing in $(\alpha,\sigma)$. 
	This yields
	\begin{align*}
		r^{N-1} \bar{v}^2 \left(\frac{\bar{u}}{\bar{v}}\right)' 
		\equiv 
		r^{N-1}(\bar{u}' \bar{v} - \bar{u}\bar{v}') 
		&> \alpha^{N-1} (\bar{u}'(\alpha) \bar{v}(\alpha) - \bar{u}(\alpha)\bar{v}'(\alpha))
		=0
		%\alpha^{N-1} (h \bar{u}(\alpha) \bar{v}(\alpha) - \bar{u}(\alpha) h\bar{v}(\alpha))
		%=0
	\end{align*}
	for any $r \in (\alpha,\sigma)$, thanks to the boundary conditions.
	Thus, the ratio ${\bar{u}}/{\bar{v}}$ is also  increasing in $(\alpha,\sigma)$. 
Since $\bar{u}(\alpha)/\bar{v}(\alpha)=1$, we deduce that 
\begin{equation}\label{eq:uv>0bar}
\bar{u}(r) > \bar{v}(r) >0 
\quad \text{for}~ r\in (\al,\sigma),
\end{equation}
see Figure~\ref{fig:lemA4}, 
cf.~\cite{M} for a related result.

\begin{figure}[ht!]
	\centering
	\includegraphics[width=0.8\linewidth]{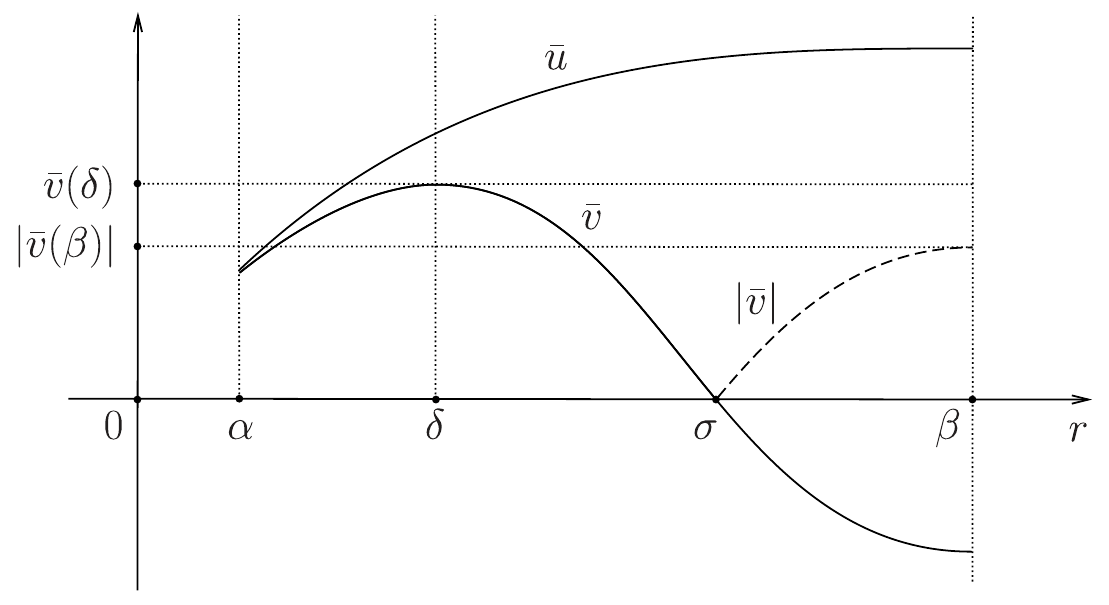}\\
	\caption{A schematic plot of $\bar{u}$ and $\bar{v}$ assuming $h>0$.}
	\label{fig:lemA4}
\end{figure}

Next, we show that  $\bar{u}>|\bar{v}|$ in the whole interval $(\alpha,\beta)$.
Noting that $\bar{v}'(\alpha)=h \geq 0$ and $\bar{v}$ is sign-changing, we get the existence of $\delta \in [\alpha,\sigma)$ such that  
  $\bar{v}'(\delta)=0$ (and $\bar{v}(\delta)>0$).
 Consequently,  recalling that $\bar{v}$ changes sign exactly once, we deduce that $\bar{v}$ is a second eigenfunction of the SL problem for the equation \eqref{eq:ode} in $(\delta,\beta)$ with the Neumann boundary conditions $\bar{v}'(\delta)=\bar{v}'(\beta)=0$.
Therefore,  \cite[Lemma~A.1]{ABD1} gives $\bar{v}'<0$ in $(\delta,\beta)$.
In particular, we see that $|\bar{v}(r)| < \max\{\bar{v}(\delta), |\bar{v}(\beta)|\}$ for $r \in (\delta,\beta)$.
At the same time, recalling that $\bar{v}$ satisfies the second equation in \eqref{eq:eqs21} and noting that $\tau_{0,2}>0$ by Lemma~\ref{lem:tau2>0},
	we apply the Sonin-Butlewski-P\'olya theorem (see, e.g., \cite[Section~7.31, p.~166]{S}) to deduce that $\bar{v}(\delta) > |\bar{v}(\beta)|$, which yields $\bar{v}(\delta) > |\bar{v}(r)|$ for $r \in (\delta,\beta)$.
 On the other hand, $\bar{u}$ is  increasing in $(\al,\be)$, see Lemma~\ref{lem:v}~\ref{eq:long:2}:\ref{eq:long:2:i}.
Combining these facts and \eqref{eq:uv>0bar}, we have
$$
	\bar{u}(r) > \bar{u}(\delta) > \bar{v}(\delta) > |\bar{v}(r)| 
 \quad 
 \text{for}~
 r \in (\delta,\beta),
	$$
which yields $\bar{u}>|\bar{v}|$ in $(\alpha,\beta)$, see Figure~\ref{fig:lemA4}.
In view of \eqref{eq:normal11}, this inequality implies that
    \begin{equation}\label{eq:inequa1111}
	u^2(r) > \left(\frac{u(\al)}{v(\al)} \, v(r)\right)^2
	\quad 
 \text{for}~
 r \in (\alpha,\beta).
	\end{equation}
	Using now the normalization \eqref{eq:normalization0}, we arrive at
	\begin{equation}\label{eq:intu>intv1}
		1=\int_\alpha^\beta r^{N-1} {u}^2 \, dr >
		\int_\alpha^\beta r^{N-1} \left(\frac{u(\al)}{v(\al)} \, v\right)^2 \, dr=\left(\frac{u(\al)}{v(\al)}\right)^2,
	\end{equation}
	which gives the desired conclusion $0<|u(\alpha)| < |v(\alpha)|$.
\end{proof}

In the following lemma, we write $\tau_{l,1}=\tau_{l,1}(h)$ and $\tau_{0,2}=\tau_{0,2}(h)$, in order to represent the dependence of eigenvalues on $h$.
\begin{lemma}\label{lem:tau}
	Let $0<\alpha<\beta$ and $l \in \mathbb{N}$. 
 %l=0???
	If $\tau_{l,1}(\hat{h}) \leq \tau_{0,2}(\hat{h})$ for some $\hat{h} \in [0,+\infty)$, then the inequality
	\begin{equation}\label{eq:muj1<mu02y}
		\tau_{l,1}(h)<\tau_{0,2}(h)
	\end{equation}
	is satisfied for any $h \in (\hat{h},+\infty]$.
\end{lemma}
\begin{proof}
    Let us consider the function 
	\begin{equation}\label{eq:fh}
		f(h)=\tau_{0,2}(h)-\tau_{l,1}(h), \quad h \in \mathbb{R}.
	\end{equation}
       This function is continuously differentiable in view of Lemma~\ref{lem:tau-derivative}.
Denote by $u$ and $v$ eigenfunctions of the SL problem \eqref{eq:ode}-\eqref{eq:bcM} corresponding to the eigenvalues $\tau_{l,1}(h)$ and $\tau_{0,2}(h)$, respectively, and normalized as in \eqref{eq:normalization1}.
	Thanks to Lemma~\ref{lem:tau-derivative}, we have 
	\begin{equation}\label{eq:tau-derivative2}
	\frac{d \tau_{l,1}(h)}{d h}
	=
	\alpha^{N-1} u^2(\alpha)
	\quad 
	\text{and}
	\quad
	\frac{d \tau_{0,2}(h)}{d h}
	=
	\alpha^{N-1} v^2(\alpha),
	\end{equation}
and hence
$$
f'(h) = \alpha^{N-1} (v^2(\alpha) - u^2(\alpha)).
$$
Thus, we infer the following implication from Lemma~\ref{lem:compareuv}: if $f(\tilde{h})\geq 0$ for some $\tilde{h} \in [0,+\infty)$, then $f'(\tilde{h}) > 0$, and hence $f' > 0$ in a neighborhood of $\tilde{h}$. 
This local monotonicity ensures that
$f(h)>0$ for any $h \in (\tilde{h},+\infty]$. Therefore, as $f(\hat{h})\ge 0$ by the assumption,  we easily conclude that $\tau_{l,1}(h)<\tau_{0,2}(h)$ for any $h\in (\hat{h},+\infty].$
\end{proof}

Now we are ready to prove Lemma \ref{lem:n+2}.
\begin{proof}[Proof of Lemma \ref{lem:n+2}]
Let $\beta>0$ be fixed.
    We first prove the chain of inequalities \eqref{eq:muj1<mu02}.
In view of \eqref{eqn:col}, it is sufficient to show that
\begin{equation}\label{eq:muj1<mu02x}
\tau_{l,1}(h,\alpha)<\tau_{0,2}(h,\alpha).
\end{equation}
We start with the existence of $\bar{h}_l^*$. 
We know from Lemma~\ref{lem:v}~\ref{eq:long:2}:\ref{eq:long:2:iii} that $\tau_{l,1}<0$ provided $h<h_1<0$ (note that $h_1$ depends on $\alpha$, $\beta$, $l$), while $\tau_{0,2}$ is always positive by Lemma~\ref{lem:tau2>0}.
Therefore, \eqref{eq:muj1<mu02x} holds at least for any $h<h_1<0$. 
Now we consider
$$
h^*=
\sup\{
\tilde{h} \in \mathbb{R}:~ 
\tau_{l,1}(h,\alpha)<\tau_{0,2}(h,\alpha) ~\text{for any}~ h \leq \tilde{h}
\}. 
$$
If $h^*<+\infty$, then we set $\bar{h}_l^* = h^*-\varepsilon$ for some $\varepsilon>0$, while if $h^* = +\infty$, then we set $\bar{h}_l^* = h^*$.
Note that in the latter case we also have $\tau_{l,1}(+\infty,\alpha)<\tau_{0,2}(+\infty,\alpha)$, which follows from Lemma~\ref{lem:tau}.

\begin{figure}[ht!]
	\centering
	\includegraphics[width=0.7\linewidth]{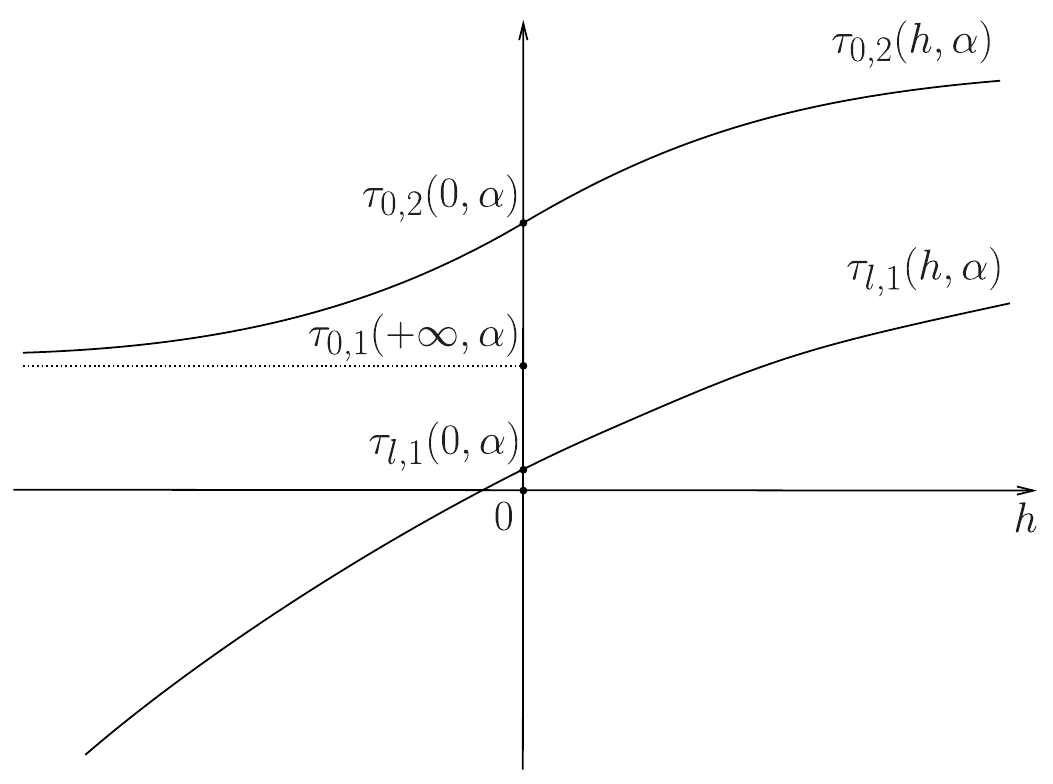}\\
	\caption{A schematic dependence of $\tau_{l,1}(h,\alpha)$ and $\tau_{0,2}(h,\alpha)$ on $h$.}
	\label{fig:proof2}
\end{figure}

Let us now prove the existence of $\alpha_l^* \in (0,\beta)$ such that \eqref{eq:muj1<mu02x} is satisfied for any $h \in \mathbb{R} \cup \{+\infty\}$ and $\alpha \in [\alpha_l^*,\beta)$, i.e., that $\bar{h}_l^* = +\infty$ 
for any $\alpha \in [\alpha_l,\beta)$.	
	Using \cite[Theorem~(a):(iii)]{EMZ} (for the convergence) and Lemma~\ref{lem:tau-derivative} (for the  monotonicity), we have
	$$
	\tau_{0,2}(h,\alpha) \searrow \tau_{0,1}(+\infty,\alpha)
	\quad \text{as}~ h \to -\infty.
	$$
Assume, for a moment, that 	\begin{equation}\label{eq:taul1<tau01}
	\tau_{l,1}(0,\alpha) <
	\tau_{0,1}(+\infty,\alpha)
	\end{equation}
	for some $\alpha \in (0,\beta)$, see Figure~\ref{fig:proof2}.
Noting that  $\tau_{l,1}(h,\alpha)$ is also increasing with respect to $h \in \mathbb{R}$, we deduce that
	$$
	\tau_{l,1}(h,\alpha)
	<
	\tau_{l,1}(0,\alpha) 
	<
	\tau_{0,1}(+\infty,\alpha)
	<
	\tau_{0,2}(h,\alpha)
	$$
	for any $h \in (-\infty,0)$.
 Lemma~\ref{lem:tau} implies then that $\tau_{l,1}(h,\alpha) < \tau_{0,2}(h,\alpha)$ also for any $h \in [0,+\infty]$.
 Thus, it remains to justify the validity of  \eqref{eq:taul1<tau01} for any $\alpha$  sufficiently close to $\beta$. 
	On one hand, it is known from, e.g., \cite[Corollary~3.12~(ii)]{KWZ2} (with a simple change of variables) that $\tau_{l,1}(0,\alpha)$ is bounded as $\alpha \to \beta-$.
	On the other hand, $\tau_{0,1}(+\infty,\alpha) \to +\infty$
	as $\alpha \to \beta-$, see, e.g., \cite[Corollary~2.5]{KWZ2}.
	Therefore, 
 there exists $\alpha_l^* \in (0,\beta)$ such that 
 \eqref{eq:taul1<tau01} (and hence \eqref{eq:muj1<mu02x}) is satisfied for any $\alpha \in [\alpha_l^*,\beta)$.

Finally, let us prove 
that for every $\alpha \in (0,\beta)$ there exists $h_2^*<0$ such that the chain of inequalities \eqref{eq:mu11<mu21<mu02} 
holds true for any $h \in [h_2^*,+\infty]$.
Again in view of \eqref{eqn:col}, it is sufficient to show that
\begin{equation}\label{eq:muj1<mu02xy}
\tau_{2,1}(h,\alpha)<\tau_{0,2}(h,\alpha).
\end{equation}
Since this inequality holds for $h=0$ (see  \cite[Lemma~2.3]{ABD1}),
for every $\alpha \in (0,\beta)$ there exists $h_2^* < 0$ such that \eqref{eq:muj1<mu02xy} remains valid for any $h \in [h_2^*,-h_2^*]$ and hence for any $h \in [h_2^*,+\infty]$ by  Lemma~\ref{lem:tau}.
\end{proof}

	\section{}\label{sec:appendix2}

For the sake of completeness, we justify that the definition
\eqref{eq:taukdef} describes the same eigenvalue $\tau_{k}(\Omega)$ as the original definition \eqref{eq:tauk}:
\begin{equation}\label{eq:tauk2}
	\tag{\ref{eq:tauk}}
	\tau_{k}(\Omega) 
	= 
	\min_{X\in\mathcal{X}_k} 
	\max_{u \in X\setminus \{0\}} \frac{\int_{\Omega} |\nabla u|^2 \, dx + h \int_{\partial B_\alpha} u^2 \,dS}{\int_{\Omega} u^2 \, dx},
	\quad
	k \geq 2,
\end{equation} 
where $\mathcal{X}_k$ is the collection of all $k$-dimensional subspaces of the Sobolev space of $\tilde{H}^1$, see \eqref{eq:sobolev}.
Denote 
\begin{equation}\label{eq:taukdef2}
	\tag{\ref{eq:taukdef}}
	\tilde{\tau}_{k}(\Omega) 
	= 
	\min_{Y\in\mathcal{Y}_{k-1}} 
	\max_{u \in Y\setminus \{0\}} \frac{\int_{\Omega} |\nabla u|^2 \, dx + h \int_{\partial B_\alpha} u^2 \,dS}{\int_{\Omega} u^2 \, dx},
	\quad
	k \geq 2,
\end{equation} 
where $\mathcal{Y}_{k-1}$ is the collection of all $(k-1)$-dimensional subspaces of $\tilde{H}^1(\Omega)$ which are $L^2(\Omega)$-orthogonal to $\mathbb{R} \phi_1$. 

\begin{lemma}\label{lem:tau=tautilde}
Let $h \in \mathbb{R} \cup \{+\infty\}$ and $\Omega = \Omega_{\textnormal{out}} \setminus \overline{B_\alpha}$ satisfy the assumption \ref{assumption}.
Then
	$\tau_{k}(\Omega) = \tilde{\tau}_{k}(\Omega)$ for any $k \geq 2$.
\end{lemma}
\begin{proof}
	Fix any $k \geq 2$ and $Y\in\mathcal{Y}_{k-1}$. It is not hard to see that the direct sum $X := Y \oplus \mathbb{R} \phi_1 \in \mathcal{X}_k$. 
	Since any $v \in Y$ is $L^2(\Omega)$-orthogonal to $\phi_1$, we have, in view of \eqref{eq:D}, that 
	$$
	\int_{\partial B_\alpha} v\phi_1 \,dS = 0
	\quad \text{and} \quad
	\intO \langle \nabla v, \nabla \phi_1 \rangle_{\mathbb{R}^N} \,dx = 0.
	$$	
 Decomposing any $u \in X$ as $u=c\phi_1 + v$ with $c \in \mathbb{R}$ and $v \in Y$, we get
	\begin{align*}
\frac{\int_{\Omega} |\nabla u|^2 \, dx + h \int_{\partial B_\alpha} u^2 \,dS}{\int_{\Omega} u^2 \, dx}
&=  \frac{c^2 \tau_1(\Omega) \int_{\Omega} \phi_1^2 \,dx +  \int_{\Omega} |\nabla v|^2 \, dx + h \int_{\partial B_\alpha} v^2 \,dS}{c^2 \int_{\Omega} \phi_1^2 \, dx + \int_{\Omega} v^2 \, dx}
	\\
 &\leq \frac{ \int_{\Omega} |\nabla v|^2 \, dx + h \int_{\partial B_\alpha} v^2 \,dS}{\int_{\Omega} v^2 \, dx},
%	\leq
%	\tilde{\tau}_{k}(\Omega)
 \end{align*}
where the last inequality follows from the definition of $\tau_1(\Om)$ and the following simple equivalence:  
$$
\frac{a+b}{d +e} \le \frac{b}{e} \iff \frac{a}{d} \le \frac{b}{e} 
$$
for $a,b \in \mathbb{R}$, $d>0$, $e> 0$.
Therefore,
$$
\tau_k(\Omega)
\leq
\max_{u \in X\setminus \{0\}} \frac{\int_{\Omega} |\nabla u|^2 \, dx + h \int_{\partial B_\alpha} u^2 \,dS}{\int_{\Omega} u^2 \, dx}
	\leq
	\max_{u \in Y\setminus \{0\}} \frac{\int_{\Omega} |\nabla u|^2 \, dx + h \int_{\partial B_\alpha} u^2 \,dS}{\int_{\Omega} u^2 \, dx}
 $$
 Minimizing over $Y\in\mathcal{Y}_{k-1}$, we obtain $\tau_{k}(\Omega) \leq \tilde{\tau}_{k}(\Omega)$. 
	Conversely, taking $X = \text{span}\{\phi_1,\dots,\phi_k\}$, where each $\phi_j$ is the $j$-th eigenfunction of \eqref{eq:D} and $\phi_j$ is orthogonal to $\phi_i$ for $i \neq j$, we can decompose $X = Y \oplus \mathbb{R} \phi_1$, where $Y = \text{span}\{\phi_2,\dots,\phi_k\}$, and so $Y \in \mathcal{Y}_{k-1}$. This yields
	$$
	\tilde{\tau}_{k}(\Omega)
	\leq
	\max_{u \in Y\setminus \{0\}} \frac{\int_{\Omega} |\nabla u|^2 \, dx + h \int_{\partial B_\alpha} u^2 \,dS}{\int_{\Omega} u^2 \, dx}
	\leq
	\max_{u \in X\setminus \{0\}} \frac{\int_{\Omega} |\nabla u|^2 \, dx + h \int_{\partial B_\alpha} u^2 \,dS}{\int_{\Omega} u^2 \, dx}
	=
	\tau_{k}(\Omega).
	$$
	This completes the proof.
\end{proof}

\begin{remark}
It is clear that the result of Lemma~\ref{lem:tau=tautilde} remains valid under much more general assumptions on $\Omega$ and the boundary conditions. 
We refrain from formulating such a general statement aiming to keep the visual simplicity, the proof being standard anyway.
\end{remark}

	\smallskip
	\noindent
	\textbf{Acknowledgments.}
	{T.V.~Anoop was supported by the MATRCIS project (MTR/2022/000222) of  SERB, India.}
	V.~Bobkov was supported in the framework of the development program of the Scientific Educational Mathematical Center of the Volga Federal District (agreement No.\ 075-02-2023-950). 
	P.~Dr\'abek was supported by 
 the Grant Agency of the Czech Republic, grant No.~22-18261S.

\end{document}